\documentclass[onefignum,onetabnum]{siamart250211}

\usepackage{amsfonts, amssymb, amsopn}
\usepackage[mathscr]{eucal}
\usepackage{bm, upgreek}
\usepackage{graphicx, subcaption, epstopdf, booktabs}
\usepackage[linesnumbered,ruled,vlined,algo2e]{algorithm2e}
\SetAlFnt{\small}

\ifpdf
  \DeclareGraphicsExtensions{.eps,.pdf,.png,.jpg}
\else
  \DeclareGraphicsExtensions{.eps}
\fi

\renewenvironment{equation*}{\[}{\]\ignorespacesafterend}

\newsiamremark{remark}{Remark}
\crefname{algocf}{algorithm}{algorithms}
\Crefname{algocf}{Algorithm}{Algorithms}

\headers{Cross approximation of function-valued tensors}{S. Budzinskiy, V. Kazeev, M. Olshanskii}

\title{Low-rank cross approximation of function-valued tensors for reduced-order modeling of parametric PDEs\thanks{S.B. was funded by the Disruptive Innovation – Early Career Seed Money funding program of the Austrian Academy of Sciences (\"OAW) and the Austrian Science Fund (FWF).
V.K. was funded in part by the Austrian Science Fund (FWF) project \href{https://doi.org/10.55776/F65}{10.55776/F65}. 
M.O. was supported in part by the U.S. National Science Foundation under award DMS-2309197.
For open-access purposes, the authors have applied a CC BY public copyright license to any author-accepted manuscript version arising from this submission.}}

\author{
Stanislav Budzinskiy\thanks{Faculty of Mathematics, University of Vienna, Kolingasse 14-16, 1090 Vienna, Austria (\email{stanislav.budzinskiy\allowbreak@univie.ac.at}).}
\and
Vladimir Kazeev\thanks{Faculty of Mathematics, University of Vienna (\email{vladimir.kazeev\allowbreak@univie.ac.at}).}
\and
Maxim Olshanskii\thanks{Department of Mathematics, University of Houston, 3551 Cullen Blvd, Houston, Texas 77204-3008, USA (\email{maolshanskiy\allowbreak@uh.edu}).}
}

\newcommand{\Clx}{\mathbb{C}}
\newcommand{\Real}{\mathbb{R}}
\newcommand{\F}{\mathbb{F}}
\newcommand{\N}{\mathbb{N}}

\newcommand{\Hilb}{\mathrm{H}}

\newcommand{\set}[3][]{#1\{ #2 : #3 #1\}}
\newcommand{\argmax}{\mathrm{argmax}}

\newcommand{\Bensor}[1]{\bm{\Tensor{#1}}}
\newcommand{\Tensor}[1]{{\mathscr{#1}}}
\newcommand{\Batrix}[1]{\bm{\Matrix{#1}}}
\newcommand{\Matrix}[1]{\mathsf{#1}}

\newcommand{\trans}{\intercal}
\newcommand{\pinv}{\dagger}

\newcommand{\krp}{\mathbin{\otimes}}
\newcommand{\rank}[2][]{\mathrm{rank} #1( #2 #1)}
\newcommand{\Rank}[3][]{\mathrm{rank_{#2}} #1( #3 #1)}
\newcommand{\Index}[1]{{#1}}
\newcommand{\cross}[3][]{\mathrm{cross}#1(#2, #3 #1)}
\newcommand{\Cross}[4][]{\mathrm{cross}_{\mathrm{#4}}#1(#2, #3 #1)}
\newcommand{\Row}[4][]{\mathrm{rows}_{#4} #1( #2, #3 #1)}
\newcommand{\hosvd}[3][]{\mathrm{hosvd}#1(#2, #3 #1)}

\newcommand{\Norm}[3][]{#1\| #2 #1\|_{\mathrm{#3}}}

\newcommand{\Dotp}[4][]{#1\langle #2, #3 #1\rangle_{\mathrm{#4}}}

\newcommand{\scup}[1]{\textup{\textsc{#1}}}

\begin{document}
\maketitle

\begin{abstract}
The paper considers function-valued tensors, viewed as multidimensional arrays with entries in an abstract Hilbert space. 
Despite the absence of the algebraic structure of a field, the geometric inner-product structure suffices to introduce the Tucker rank, higher-order SVD, and Tucker-cross decomposition for function-valued tensors.
An adaptive cross-approximation algorithm is developed to compute low-rank approximations of such tensors.
The framework is motivated by, and applied to, model order reduction of the parameter-to-solution map for a parametric PDE.
The resulting reduced-order model can be interpreted as an encoder-decoder scheme with a nonlinear encoder and a multilinear decoder.
The performance of the proposed non-intrusive approximation method is demonstrated in numerical examples for two nonlinear parametric PDE systems.
\end{abstract}



\section{Introduction}
The need for solving a parametric PDE for varying parameter values arises frequently in inverse problems, optimal control, uncertainty quantification, and digital-twin applications~\cite{tarantola2005inverse, biegler2003large, torzoni2024digital}.
\emph{Reduced-order modeling} (ROM) aims to extract the similarity of the PDE instances corresponding to distinct parameter values and to exploit it for reducing the cost of solving the PDE.  
A wide spectrum of ROM techniques exists, spanning from physics-based projection methods, such as POD-Galerkin and reduced basis approaches~\cite{berkooz1993proper, hesthaven2016certified, quarteroni2015reduced}, to fully data-driven models enabled by machine learning and nonlinear regression techniques~\cite{kutz2016dynamic, hesthaven2018non, willard2020integrating}.

This paper applies methods of multilinear (tensor) algebra to develop a method for the data-driven ROM of parametric PDEs.
The proposed ROM method approximates the map from the parameter domain to a \emph{quantity of interest} (QoI), which may be either a (vector-valued) functional of the PDE solution or the solution itself.
The key dimension-reduction technique underlying the proposed method is a low-rank cross approximation of function-valued tensors, which leads to a data-driven ROM based on the non-intrusive sparse sampling of the parametric solution manifold.

Function-valued tensors are currently not as well understood as their scalar-valued counterparts, possibly due to the absence of the algebraic structure available to tensors over a field.
In this paper, we show that arrays with values in a Hilbert (or Euclidean) space admit the generalization of basic notions and techniques of low-rank tensor approximation: we introduce the Tucker rank, the higher-order singular value decomposition (HOSVD), and the Tucker-cross decomposition for such function-valued tensors.
We also present an adaptive cross-approximation algorithm for computing low-rank approximations of such tensors based on subsets of their entries.

Together with the novel notions and techniques of low-rank approximation for function-valued tensors, we present herein their application to model-order reduction for parametric PDEs.
In particular, we show that the resulting ROM method can be interpreted as an encoder-decoder scheme with a \emph{nonlinear encoder} and a \emph{multilinear decoder}.
This can be contrasted with the well-known POD-Galerkin and {reduced basis} (RB) approaches, which are based on the classical low-rank matrix approximation and can be viewed as linear encoder-decoder schemes performing projection onto linear latent spaces.
At the same time, in comparison with autoencoder-based ROM methods based on neural networks, the tensor-based methods remain tractable, and their interpolation properties appear amenable to rigorous analysis.

\subsection{Outline}
Since function-valued tensors do not routinely appear in the literature, and the ideas behind their application to ROM can be rather intricate, the exposition in this paper begins with the matrix case in \Cref{sec2}.
We first recall a few well-known results on the cross approximation of real-valued matrices and then proceed to matrices over Hilbert spaces.
Understanding the notions of rank, cross decomposition, and their application to ROM in the matrix setting makes their subsequent extension to function-valued tensors in \Cref{sec:cross_tucker} more straightforward.
\Cref{sec:abc} introduces an algorithm for adaptive function-valued cross approximation in the Tucker format, which in the present paper serves as the main model-reduction technique.
Finally, \Cref{sec:numer} illustrates the approach with numerical experiments, in which the proposed ROM method is applied for the approximation of parameter-to-solution maps:
a system describing the slow motion of a non-Newtonian incompressible fluid with parametrized rheology and boundary conditions, and a parametric Monge--Amp\`ere equation.

\subsection{Relation to prior work on function-valued tensors}
\label{subsec:fvt}
Vectors, matrices, and tensors with function-valued entries have been considered previously, though in settings different from ours \cite{shustin2022semi, mor2025quasitubal}.
In particular, tensors with entries in a \emph{reproducing kernel Hilbert space} (RKHS) were studied in~\cite{townsend2015continuous, larsen2024tensor, han2024guaranteed, tang2026tensor},
where an order-$d$ function-valued tensor over an RKHS is treated as an order-$(d+1)$ real-valued tensor with one continuous dimension, which is approximated by pointwise sampling with respect to all dimensions, including the continuous one.
Such point evaluation with respect to the continuous dimension is generally not well-defined outside of the RKHS setting and can be replaced with more general measurements.
In the context of finite-element PDE solutions, the natural measurements are provided by a finite-element basis, and all of these measurements are queried simultaneously.
A cross-approximation algorithm following the prior work would use only a small subset of measurements during pivot selection and approximation update, discarding the rest.

In contrast, the framework we propose is based on accessing entries of a function-valued tensor exclusively as elements of the underlying Hilbert space, which can be arbitrary.
Our adaptive cross approximation is then driven by the inner products of the queried functions, which are treated as the most essential information contained in the data and leveraged at full capacity.
The framework is non-intrusive, applies to more general function-valued tensors, and is suitable for typical PDE settings.

To elucidate the distinction between our approach and prior work, consider a function-valued vector (e.g., produced by a one-parameter PDE).
Adaptive cross approximation in the sense of \cite{townsend2013extension} then operates precisely as a cross-approximation algorithm for real-valued matrices, or equivalently as an incomplete LU decomposition with row and column pivoting.
Meanwhile, our method amounts to an incomplete QR decomposition with column pivoting in this case.

For a Hilbert space $\Hilb$ and $n_1,\ldots,n_d \in \N$, the space $\Hilb^{n_1 \times \cdots \times n_d}$ of tensors of size $n_1 \times \cdots \times n_d$ with values in $\Hilb$ is isometrically isomorphic to a tensor-product Hilbert space $\Hilb \krp \Real^{n_1} \krp \cdots \krp \Real^{n_d}$.
As such, it is covered by the treatment of tensor products of arbitrary Banach spaces given in~\cite{hackbusch2019tensor};
in particular, the HOSVD approximation naturally generalizes to tensor products of arbitrary Hilbert spaces~\cite[\S~10]{hackbusch2019tensor}.
In contrast, our particular setting of $\Hilb^{n_1 \times \cdots \times n_d} \simeq \Hilb \krp \Real^{n_1} \krp \cdots \krp \Real^{n_d}$ allows for generalizing not only the HOSVD approximation but also the Tucker-cross approximation.

\subsection{Related work on tensor methods for parametric PDEs}
\label{sec:related_pde}
The \emph{low-rank tensor decomposition} (LRTD) Galerkin ROM for parametric dynamical systems was introduced in~\cite{mamonov2022interpolatory} and further developed and analyzed in~\cite{mamonov2024tensorial, mamonov2025priori, olshanskii2025approximating}.
In these works, the POD component of the Galerkin ROM was replaced by the LRTD to provide a parameter-specific local basis for approximating the system dynamics in a reduced latent space.
Reduced-order models for parametric PDEs based on the low-rank completion of tensors sampled by slices was studied in~\cite{mamonov2024slice}.
That approach can also be interpreted as finding a low-rank approximation of a vector-valued tensor from sparse sampling; however, it does not exploit the Hilbert structure of the space to which the tensor entries belong. 

For the approximation of scalar-valued QoI of parametric PDEs, cross approximation was used in \cite{ballani2015hierarchical}.
Similar approaches were developed in \cite{bigoni2016spectral, gorodetsky2019continuous, strossner2024approximation}.

Parameter-to-solution maps can be expanded with respect to fixed bases in the parameter and spatial domains, with the resulting coefficient tensor approximated with LRTD~\cite{khoromskij2011tensor}.
Following this line of research:
(i) Richardson iteration with adaptive rank truncation was proposed in~\cite{kressner2011low, bachmayr2018parametric} to obtain an approximate least-squares solution of the parametric PDE~\cref{eq:model_pde} below, whose coefficient tensor is represented in a low-rank hierarchical format;
(ii) the tensor-train (TT) and the hierarchical tensor representations were employed in~\cite{khoromskij2010quantics, kressner2011low, dolgov2018direct, dolgov2019hybrid}, where discretization of~\cref{eq:model_pde}, with collocation in the parameter space, yields block-diagonal linear systems with low-rank structure across the blocks, the TT structure being imposed analytically in~\cite{khoromskij2010quantics, kressner2011low, dolgov2018direct} and via the TT cross approximation of the parametric coefficient in~\cite{dolgov2019hybrid};
(iii) a stochastic version of~\cref{eq:model_pde} was studied in~\cite{dolgov2015polynomial}, where the Karhunen–Lo\`eve expansion of the stochastic diffusion coefficient, followed by block TT cross approximation, was used for a stochastic formulation of the problem.
The methods described above rely on the linearity of~\cref{eq:model_pde}, as they focus on assembling and solving tensor-structured linear systems; in contrast, our approach also applies to nonlinear PDEs.

The non-intrusive method of~\cite{eigel2019non, eigel2019variational} treats the parametric PDE as a black box, making no assumptions about its structure and relying solely on sampled solutions---much like our approach.
Given a set of Monte Carlo samples,~\cite{eigel2019non, eigel2019variational} fits a low-rank TT decomposition of the coefficient tensor to these data in a least-squares sense using alternating steepest descent.
Our method differs in two essential aspects:
(i) adaptive cross approximation is an “active-learning’’ technique that \emph{generates samples}, whereas~\cite{eigel2019non, eigel2019variational} performs “passive learning’’ on \emph{given samples};
(ii) our approach is based on \emph{interpolation}, while that of~\cite{eigel2019non, eigel2019variational} relies on \emph{least-squares approximation}.

A block TT-cross algorithm was developed in~\cite{dolgov2015polynomial} to construct a block TT decomposition from $N$-blocked samples.
From a vector-valued perspective, this algorithm implicitly equips the space~$\Real^N$ with the standard inner product.
Hence, following our results in \Cref{sec:cross_tucker}, the resulting approximation can be regarded as a ``mathematically correct'' cross approximation for that particular Hilbert space.
The associated inner product is, however, not suitable in the context of a parametric PDE.
In the present paper, the formalism of function-valued tensors incorporates the Hilbert-space structure into all computations and can therefore be viewed as a way to extend the block TT-cross algorithm to parameter-to-solution maps.
Its use in~\cite{dolgov2015polynomial} is limited to assembling the linear system.
\section{Low-rank cross approximation as reduced-order modeling}
\label{sec2}
Our motivating example is a diffusion equation with a coefficient depending on two parameters $\alpha, \beta \in [0,1]$:
\begin{equation}
\label{eq:model_pde}
    \mathrm{div}\left(a(\Matrix{x}; \alpha,\beta)\nabla u(\Matrix{x}; \alpha,\beta)\right) = f(\Matrix{x}),
    \quad \Matrix{x} \in \Omega, \quad
    \left. u \right|_{\partial \Omega} = 0.
\end{equation}
Equation \cref{eq:model_pde} defines a parametric family of mathematical models.
In numerical optimization or inverse engineering, one often seeks to select a single model from this family according to certain criteria, i.e., by minimizing a cost functional.

These criteria may depend not on the solution of \cref{eq:model_pde} itself, but on a QoI computed from it.
In the simplest case, the QoI is a scalar-valued function $\eta : [0,1]^2 \to \Real$, such as the spatial average of the solution:
\begin{equation*}
    \eta(\alpha, \beta) = \int_{\Omega} u(\Matrix{x}; \alpha, \beta) d\Matrix{x}.
\end{equation*}

In an optimization loop, the QoI may need to be evaluated at many parameter values, requiring repeated numerical solutions of \cref{eq:model_pde} and leading to high computational cost.
A common strategy to mitigate this is to construct a ROM, i.e., to approximate $\eta$ in a data-driven fashion.

There exist many data-driven approximation methods, including sparse grids \cite{bungartz2004sparse}, kernel methods \cite{wendland2004scattered}, sparse polynomial approximation \cite{adcock2022sparse}, and neural networks \cite{devore2021neural}.
Here, we focus on low-rank ROMs, constructed for example as follows:
\begin{enumerate}
    \item Evaluate $\eta$ on a Cartesian $m \times n$ grid in the parameter domain and use these values to define an interpolant of $\eta$.
    \item Compute a rank-$r$ approximation of the matrix $\Matrix{A} = [\eta(\alpha_i, \beta_j)] \in \Real^{m \times n}$ via truncated singular value decomposition (SVD).
\end{enumerate}

This yields an “approximate interpolant” of the QoI.
The low-rank approximation \emph{compresses} the exact interpolant from $mn$ parameters to $O(mr + nr)$.
Its main drawback is that the full matrix $\Matrix{A}$ must be evaluated and stored, which is prohibitive when $mn$ is large and the PDE is expensive to solve.
Since the resulting low-rank ROM is characterized by $O(mr + nr)$ parameters, one may seek to construct it using only $O(mr + nr)$ function evaluations (i.e., PDE solves).
In the case of two parameters and a scalar QoI, the matrix cross approximation may serve this purpose.

\subsection{Matrix cross approximation}
\label{subsec:scalar_cross}
The following \emph{interpolation identity} can serve as the basis for an efficient low-rank ROM.
Every matrix $\Matrix{A} \in \Real^{m \times n}$ of rank $r$ contains at least one invertible submatrix $\Matrix{A}(I_A, J_A) \in \Real^{r \times r}$, and any such a submatrix allows for exactly recovering $\Matrix{A}$ from its $(m + n - r)r$ entries by \emph{cross interpolation}:
\begin{equation}
\label{eq:interpolation_matrix}
    \Matrix{A} = \Matrix{A}(:, J_A)\, \Matrix{A}(I_A, J_A)^{-1}\, \Matrix{A}(I_A, :).
\end{equation}

The right-hand side of \cref{eq:interpolation_matrix} extends naturally to arbitrary nonempty index sets\footnote{Which are treated as ordered tuples of increasing indices when used to slice a matrix.} $I \subseteq [m]$ and $J \subseteq [n]$ at the cost of possibly losing equality, resulting in \emph{cross approximation} by cross interpolation.
Furthermore, the invertibility requirement can be relaxed by replacing the inverse with the pseudoinverse:
\begin{equation}
\label{eq:cross_matrix}
    \cross{\Matrix{A}}{I, J} = \Matrix{A}(:, J)\, \Matrix{A}(I, J)^{\pinv}\, \Matrix{A}(I, :) \in \Real^{m \times n}.
\end{equation}
The cross approximation $\Matrix{B} = \cross{\Matrix{A}}{I, J}$ inherits the interpolatory property of \cref{eq:interpolation_matrix}: it satisfies $\Matrix{B}(I,J) = \Matrix{A}(I,J)$ and exactly recovers $\Matrix{A}$ whenever $\rank{\Matrix{A}(I,J)} = \rank{\Matrix{A}}$; see, e.g., \cite{hamm2020perspectives}.

When $\rank{\Matrix{A}(I,J)} < \rank{\Matrix{A}}$, the approximation quality of \cref{eq:cross_matrix} depends on the choice of $I$ and $J$.
The \emph{adaptive cross approximation} (ACA) algorithm \cite{bebendorf2000approximation} constructs these sets iteratively, selecting new index pairs based on the largest residual entries.
A more detailed discussion of the ACA follows in \Cref{sec:abc}.

A cross approximation computed by the ACA can be used to define a low-rank ROM for a scalar-valued QoI.
When the ACA employs rook pivoting to select index pairs, it requires access to $O(m+n)$ matrix entries per iteration—equivalently, this many PDE solves in our setting.
Hence, $O(mr + nr)$ QoI samples suffice to construct a rank-$r$ ROM, as desired.

\subsection{Bi-parametric PDE with vector- and function-valued QoI}
Let us return to the parametric PDE \cref{eq:model_pde} and consider a vector-valued QoI.
For example, this can be a vector-function $\eta : [0,1]^2 \to \Real^k$ that computes the first $k$ moments of the solution.
The ``ultimate'' vector-valued QoI is the full-order-model solution of the PDE itself with $k$ being the number of degrees of freedom.
Even more generally, we can consider a QoI with values in an abstract function space.
Under standard regularity assumptions on the diffusion coefficient and the right-hand side, the weak solution of~\cref{eq:model_pde} belongs to the Sobolev space $\Hilb = \Hilb^{1}_0(\Omega)$; see \cite{brezis2011functional}.
If these assumptions are satisfied for every parameter value, one can consider the \emph{parameter-to-solution} map $\eta : [0,1]^2 \to \Hilb$ as the function-valued QoI.

Attempting to construct a ROM of the parameter-to-solution map $\eta$ along the same lines as described above, one ends up with an $m \times n$ array whose entries are elements of $\Hilb$.
Below, we give meaning to low-rank approximation and, more importantly, cross approximation for such arrays---leading to a sample-efficient low-rank ROM for vector-valued QoIs such as the parameter-to-solution map.

\subsection{Function-valued matrix cross approximation}
Consider a general setting where $\Hilb$ is a Hilbert space over $\F \in \{ \Real, \Clx \}$.
Two-dimensional arrays $\Batrix{A} \in \Hilb^{m \times n}$, referred to as \emph{Bochner matrices}, were studied in~\cite{budzinskiy2025matrices}.
They can be transposed $\Batrix{A}^{\trans} \in \Hilb^{n \times m}$ and multiplied with matrices $\Matrix{B} \in \F^{k \times m}$ and $\Matrix{C} \in \F^{n \times l}$ according to the standard formulas to give $\Matrix{B} \Batrix{A} \Matrix{C} \in \Hilb^{k \times l}$, but two function-valued matrices cannot be multiplied even if they have conforming sizes.
The \emph{adjoint} of $\Batrix{A}$ is defined as a linear operator $\Batrix{A}^{\ast} : \Hilb^{m} \to \F^{n}$,
\begin{equation*}
    \Batrix{A}^\ast \Batrix{b} = \begin{bmatrix}
        \Dotp{\Batrix{b}}{\Batrix{A}(:,1)}{\ell_2(\Hilb)} &
        \cdots &
        \Dotp{\Batrix{b}}{\Batrix{A}(:,n)}{\ell_2(\Hilb)}
    \end{bmatrix}^\trans, \quad \Batrix{b} \in \Hilb^m,
\end{equation*}
where $\Dotp{\Batrix{b}}{\Batrix{a}}{\ell_2(\Hilb)} = \sum_{i = 1}^{m} \Dotp{\Batrix{b}(i)}{\Batrix{a}(i)}{\Hilb}$ is an inner product in $\Hilb^m$.
By convention, the adjoint is applied columnwise to $\Batrix{B} \in \Hilb^{m \times k}$, so that $\Batrix{A}^{\ast} \Batrix{B} \in \F^{n \times k}$.
A function-valued matrix $\Batrix{Q} \in \Hilb^{m \times n}$ is said to have orthonormal columns if $\Batrix{Q}^{\ast} \Batrix{Q} = \Matrix{I}_{n}$.

The \emph{column-rank} of $\Batrix{A}$, $r = \Rank{c}{\Batrix{A}}$, is defined as the number of its linearly independent columns.
Function-valued matrices admit an SVD, $\Batrix{A} = \Batrix{U} \Matrix{\Sigma} \Matrix{V}^{\ast}$, where $\Batrix{U} \in \Hilb^{m \times r}$ is a function-valued matrix with orthonormal columns, $\Matrix{\Sigma} \in \Real^{r \times r}$ is a diagonal matrix with positive nonincreasing entries, and $\Matrix{V} \in \F^{n \times r}$ is a matrix with orthonormal columns.
Such an SVD can be computed via a QR decomposition of $\Batrix{A}$ followed by computing an SVD of the scalar-valued triangular factor.
The \emph{pseudoinverse} of a nonzero function-valued matrix $\Batrix{A}$ is a linear operator $\Batrix{A}^{\pinv} : \Hilb^{m} \to \F^{n}$ given by $\Batrix{A}^\pinv = \Matrix{V} \Matrix{\Sigma}^{-1} \Batrix{U}^\ast$,
and it is zero when $\Batrix{A}$ is zero.

We have introduced the necessary definitions and proceed to the cross approximation of function-valued matrices.
The interpolation identity \cref{eq:interpolation_matrix} extends as follows:
\begin{equation}
\label{eq:precross_batrix}
    \Batrix{A} = \underbrace{\Batrix{A}(:, J)}_{\Hilb^{m \times |J|}} \cdot \underbrace{\Batrix{A}(I, J)^{\pinv} \Batrix{A}(I, :)}_{\F^{|J| \times n}}
\end{equation}
if and only if $\Rank{c}{\Batrix{A}(I, J)} = \Rank{c}{\Batrix{A}}$.
Since $\Batrix{A}(I, J)^{\pinv}$ is a linear operator, it cannot be left as a stand-alone factor in the decomposition, so the size of the index set $|I|$ ends up being ``hidden.''
This is the first distinction from \cref{eq:interpolation_matrix} and \cref{eq:cross_matrix}.

The second one arises as we consider minimal $J$ so that the columns of $\Batrix{A}(:, J)$ are linearly independent.
When $\Hilb = \F$, it is not possible to further compress these columns.
But for general function-valued matrices, the \emph{row-rank} (the column-rank of the transpose) is not necessarily equal to the column-rank, so we can apply \cref{eq:precross_batrix} to $\Batrix{A}(:, J)^{\trans}$.
This observation leads to the definition of cross approximation for function-valued matrices:
\begin{equation}
\label{eq:cross_batrix}
    \cross{\Batrix{A}}{I, J} = \underbrace{\Big\{ [\Batrix{A}(I,J)^{\trans}]^{\pinv} \Batrix{A}(:,J)^{\trans} \Big\}^{\trans}}_{\F^{m \times |I|}} \cdot \underbrace{\Batrix{A}(I, J)}_{\Hilb^{|I| \times |J|}} \cdot \underbrace{\Batrix{A}(I,J)^{\pinv} \Batrix{A}(I,:)}_{\F^{|J| \times n}}.
\end{equation}
Note that transposition and pseudoinversion do not commute for function-valued matrices, hence the form of the left factor.
The cross approximation $\Batrix{B} = \cross{\Batrix{A}}{I, J}$ enjoys  interpolation properties similar to \cref{eq:cross_matrix}: the submatrix $\Batrix{B}(I,J) = \Batrix{A}(I,J)$ is always preserved, and $\Batrix{B} = \Batrix{A}$ whenever the column-rank and row-rank of $\Batrix{A}(I,J)$ coincide with those of $\Batrix{A}$.

The ACA algorithm was extended to function-valued matrices in \cite{budzinskiy2025matrices}, providing a sample-efficient way to select the index sets $I$ and $J$ for \cref{eq:cross_batrix}.

\begin{remark}
\label{remark:ip}
If we fix $I$ and $J$ and change the inner product, the cross approximation \cref{eq:cross_batrix} will change.
This property highlights that function-valued cross approximation takes the geometry of the Hilbert space into account.
\end{remark}

\begin{remark}
\label{remark:onb}
Let $\{ \phi_\lambda \}_{\lambda = 1}^{\Lambda}$ be an orthonormal basis in $\mathrm{span}\,\{ \Batrix{A}(i,j) \}_{i=1}^m{\,}_{j=1}^n$, then the expression \cref{eq:cross_batrix} can be reformulated in terms of a tensor $\Tensor{A} \in \F^{m \times n \times \Lambda}$  of expansion coefficients and its two unfoldings.
Indeed, it holds by definition that
\begin{alignat*}{3}
    &
    \Matrix{F}_1
    &
    \, = \,
    &
    \Big\{ [\Batrix{A}(\Index{I}, \Index{J})^\trans]^{\pinv} \Batrix{A}(:, \Index{J})^\trans \Big\}^\trans
    &
    \, = \;
    &
    \Matrix{A}_{(1)} \, (\Matrix{S}_{2}^\trans \krp \Matrix{I}_{\Lambda})
    \Big[ \Matrix{S}_{1} \, \Matrix{A}_{(1)} (\Matrix{S}_{2}^\trans \krp \Matrix{I}_{\Lambda}) \Big]^\pinv
    ,
    \\
    &
    \Matrix{F}_2
    &
    \, = \,
    &
    \Big\{ \Batrix{A}(\Index{I}, \Index{J})^{\pinv} \Batrix{A}(\Index{I}, :) \Big\}^\trans
    &
    \, = \;
    &
    \Matrix{A}_{(2)} \, (\Matrix{S}_{1}^\trans \krp \Matrix{I}_{\Lambda})
    \Big[ \Matrix{S}_{2} \, \Matrix{A}_{(2)} (\Matrix{S}_{1}^\trans \krp \Matrix{I}_{\Lambda}) \Big]^\pinv,
\end{alignat*}
where $\Matrix{S}_{1} = \Matrix{I}_m(\Index{I}, :)$ and $\Matrix{S}_{2} = \Matrix{I}_n(\Index{J}, :)$ are row-selection matrices.
The function-valued cross-approximant $\Batrix{B} \in \Hilb^{m \times n}$ and its coefficient tensor $\Tensor{B} \in \F^{m \times n \times \Lambda}$ then satisfy
\begin{equation*}
    \Tensor{B} = \Tensor{A}(\Index{I}, \Index{J}, :) \times_1 \Matrix{F}_1 \times_2 \Matrix{F}_2
    \quad\text{and}\quad
    \Batrix{B}(i,j) = \sum_{\lambda = 1}^{\Lambda} \Tensor{B}(i,j,\lambda) \, \phi_{\lambda}
    \, .
\end{equation*}
Moreover, $\Norm{\Batrix{A} - \Batrix{B}}{\ell_2(\Hilb)} = \Norm{\Tensor{A} - \Tensor{B}}{\ell_2}$.
See \Cref{sec:cross_tucker} for the tensor notation and note that this reformulation admits an extension to the Tucker-cross approximation of function-valued tensors introduced therein (see also \Cref{subsec:stokes}).
\end{remark}

\subsection{Reduced-order modeling}
\label{subsec:rom_matrix}
The discussion of function-valued cross approximation has been purely algebraic so far; let us now view it from a different perspective.
For both parameters, consider one-dimensional grids with nodes ${ \alpha_1, \ldots, \alpha_m }$ and ${ \beta_1, \ldots, \beta_n }$, and fix interpolation basis functions ${ \varphi_1, \ldots, \varphi_m }$ and ${ \psi_1, \ldots, \psi_n }$ satisfying $\varphi_i(\alpha_j) = \delta_{ij}$ for $i,j \in [m]$ and $\psi_i(\beta_j) = \delta_{ij}$ for $i,j \in [n]$.
The first (implicit) step in constructing the low-rank ROM is interpolation:
\begin{equation*}
    \eta(\alpha, \beta) \approx \sum_{i = 1}^{m} \sum_{j = 1}^{n} \eta(\alpha_i, \beta_j) \varphi_i(\alpha) \psi_{j}(\beta),
\end{equation*}
which we can rewrite using the function-valued matrix $\Batrix{A}$ of samples as
\begin{equation*}
    \eta(\alpha, \beta) \approx \begin{bmatrix}
        \varphi_{1}(\alpha) & \cdots & \varphi_{m}(\alpha)
    \end{bmatrix} \Batrix{A} \begin{bmatrix}
        \psi_{1}(\beta) & \cdots & \psi_{n}(\beta)
    \end{bmatrix}^\trans.
\end{equation*}
Once the grid nodes and basis functions have been chosen, the offline (``learning'') phase of ROM is carried out with an ACA-like algorithm.
This results in a cross approximation of the form
\begin{equation*}
    \Batrix{A} \approx \Matrix{F} \cdot \Batrix{A}(I,J) \cdot \Matrix{P}^{\trans}, \quad \Matrix{F} \in \Real^{m \times |I|}, \quad \Matrix{P} \in \Real^{n \times |J|},
\end{equation*}
and in the associated cross-approximation-based ROM given by
\begin{equation}
\label{eq:lr_rom}
    \hat{\eta}(\alpha, \beta) = \bigg\{ \begin{bmatrix}
        \varphi_{1}(\alpha) & \cdots & \varphi_{m}(\alpha)
    \end{bmatrix} \Matrix{F} \bigg\} \cdot \Batrix{A}(I,J) \cdot \bigg\{ \begin{bmatrix}
        \psi_{1}(\beta) & \cdots & \psi_{n}(\beta)
    \end{bmatrix} \Matrix{P} \bigg\}^\trans.
\end{equation}

\paragraph{Encoder-decoder scheme}
In the fields of ROM and operator learning, the respective models are often represented as compositions of two building blocks \cite{bhattacharya2021model}:
\begin{enumerate}
    \item an encoder that maps the input into its latent representation;
    \item a decoder that maps the latent representation to the output space.
\end{enumerate}
For the low-rank ROM \cref{eq:lr_rom}, the \emph{encoder} is given by the interpolation step
\begin{equation*}
    (\alpha, \beta) \mapsto \{\Matrix{\upphi}, ~\Matrix{\uppsi}\}:= \{
        \varphi_{1}(\alpha)  \cdots  \varphi_{m}(\alpha),~ \psi_{1}(\beta)  \cdots  \psi_{n}(\beta)
    \} \in \Real^{m + n}
\end{equation*}
followed by the transformation 
\begin{equation*}
        \{\Matrix{\upphi}, ~\Matrix{\uppsi}\} \mapsto \{
        \Matrix{\upphi^{\trans} F},~  \Matrix{\uppsi^{\trans} P}
    \} \in \Real^{|I| + |J|}.
\end{equation*}
This order-reduction transformation is ``learned'' by the ACA-like algorithm: the matrices $\Matrix{F}$ and $\Matrix{P}$ of the cross component \cref{eq:cross_batrix} depend on the index sets $I$ and $J$, and also on the inner product of the Hilbert space.

The \emph{decoder} computes the output via
\begin{equation*}
    \{ \Matrix{\upphi^{\trans} F},~\Matrix{\uppsi^{\trans} P}\} \mapsto \Matrix{\upphi^{\trans} F} \cdot \Batrix{A}(I,J) \cdot (\Matrix{\uppsi^{\trans} P})^{\trans} \in \Hilb.
\end{equation*}
This mapping is also ``learned'' by the ACA-like algorithm and depends on the selected index sets (their choice, in turn, depends on the inner product).

The resulting low-rank ROM \cref{eq:lr_rom} consists of a \emph{nonlinear} encoder and a multilinear decoder.

\paragraph{Interpolation on a subgrid}
The interpolation properties of function-valued cross approximations \cref{eq:cross_batrix} guarantee that the low-rank ROM \cref{eq:lr_rom} is exact on a subgrid: 
\begin{equation*}
    \hat{\eta}(\alpha_i, \beta_j) = \eta(\alpha_i, \beta_j), \quad i \in I, \quad j \in J.
\end{equation*}
This property suggests that the ``learning'' phase can be interpreted as an adaptive search of an ``important'' coarse subgrid within a fine grid and the coarsening of the interpolation basis.
Both steps depend on the specific inner product of the Hilbert space (\Cref{remark:ip}).

\paragraph{Comparison with the reduced-basis method}
The RB method in ROM~\cite{quarteroni2015reduced} seeks to approximate the parametric solution manifold by a low-dimensional linear subspace.
However, it does not provide an explicit formula for projecting the solution onto this subspace; instead, one needs to solve a Galerkin-projected PDE.
For a linear PDE, both the encoder and decoder of such POD-Galerkin ROM are linear. 

The low-rank ROM \cref{eq:lr_rom} also identifies a linear subspace of $\Hilb$ spanned by the entries of $\Batrix{A}(I,J)$, and directly provides the coefficients of their linear combination.
Notably, these coefficients do not arise from a projection onto the subspace but are obtained through a nonlinear encoding step.
Crucially, POD-Galerkin ROM is \emph{intrusive}, whereas the proposed low-rank ROM is \emph{non-intrusive}, i.e., it evaluates an explicit approximation of the parameter-to-solution map without an auxiliary or projected PDE solve during the online phase.

It remains possible within our framework to subsequently solve a projected PDE, which makes our encoder physics-informed.
This is realized, for example, in the interpolatory tensor ROMs developed in~\cite{mamonov2022interpolatory}, which, however, employ a full tensor HOSVD for the reduction step.

Following the discussion in \Cref{subsec:fvt}, we can further characterize the low-rank ROM \cref{eq:lr_rom} as being aware of the Cartesian-product structure of the parameter space.
Whereas when there is a single parameter, or the parameter space is flattened, it reduces to an RB-like method.
Therefore, our approach can also be perceived as a non-trivial extension of the RB framework to Cartesian-product parameter spaces. 
\section{Cross approximation of function-valued tensors}
\label{sec:cross_tucker}
In this section, we extend the previous discussion from two parameters and function-valued matrices to $d \geq 2$ parameters and \emph{function-valued tensors}.
This constitutes the main theoretical and conceptual contribution of the article.
To simplify presentation, while still conveying our core ideas, we will focus on the Tucker decomposition.
For an introduction to tensor decompositions, see, e.g., \cite{ballard2025tensor}.
The following lemma will be used repeatedly.

\begin{lemma}
\label{lemma:ranks}
Let $\Batrix{A} \in \Hilb^{m \times n}$ and $r = \Rank{c}{\Batrix{A}}$, and let $I \subseteq [m]$ and $J \subseteq [n]$.
\begin{enumerate}
    \item There exist $J_{\Batrix{A}} \subseteq [n]$ of cardinality $|J_{\Batrix{A}}| = r$ and $\Matrix{B} \in \F^{r \times n}$ such that
    \begin{equation*}
        \Batrix{A} = \Batrix{A}(:, J_{\Batrix{A}}) \Matrix{B}, \quad \Rank{c}{\Batrix{A}(:, J_{\Batrix{A}})} = r, \quad \rank{\Matrix{B}} = r.  
    \end{equation*}
    \item For all $\Matrix{C} \in \F^{k \times m}$ and $\Matrix{D} \in \F^{n \times l}$, $\Rank{c}{\Matrix{C} \Batrix{A} \Matrix{D}} \leq r$. 
    In particular,
    \begin{equation*}
        \Rank{c}{\Batrix{A}(I, J)} \leq r.
    \end{equation*}
    \item $\Batrix{A} = \Batrix{A}(:,J) \Batrix{A}(I,J)^{\pinv} \Batrix{A}(I,:)$ if and only if $\Rank{c}{\Batrix{A}(I, J)} = r$.
\end{enumerate}
\end{lemma}
\begin{proof}
See Lemma~3.1, Proposition~4.1, Lemma~5.1, and Proposition~5.2 in \cite{budzinskiy2025matrices}.
\end{proof}

\subsection{Tucker decomposition}
Let $\Bensor{A} \in \Hilb^{n_1 \times \cdots \times n_d}$.
The \emph{mode-$k$ unfolding} of $\Bensor{A}$ is a function-valued matrix $\Batrix{A}_{(k)}$ of size $n_k \times \prod_{l \neq k} n_l$ whose columns are the mode-$k$ fibers of $\Bensor{A}$ arranged as
\begin{equation*}
    \Batrix{A}_{(k)}(i_k, \overline{i_1, \ldots, i_{k-1}, i_{k+1}, \ldots, i_d}) = \Bensor{A}(i_1, \ldots, i_d).
\end{equation*}
By $\overline{i_1, \ldots, i_{k-1}, i_{k+1}, \ldots, i_d}$, we denote a ``long'' index written in the big-endian notation. Let $\Matrix{B}_k \in \F^{m_k \times n_k}$.
The \emph{mode-$k$ product} of $\Bensor{A}$ with $\Matrix{B}_k$ is a function-valued tensor $\Bensor{C} = \Bensor{A} \times_k \Matrix{B}_k$ with the mode-$k$ unfolding given by $\Batrix{C}_{(k)} = \Matrix{B}_k \Batrix{A}_{(k)}$.
This product satisfies the standard commutativity property
\begin{equation*}
    (\Bensor{A} \times_k \Matrix{B}_k) \times_l \Matrix{B}_l = (\Bensor{A} \times_l \Matrix{B}_l) \times_k \Matrix{B}_k, \quad l \neq k.
\end{equation*}
The mode-$k$ unfolding of $\Bensor{C} = \Bensor{A} \times_1 \Matrix{B}_1 \times_2 \cdots \times_d \Matrix{B}_d$ can be represented as\footnote{The order of Kronecker products would be reversed if little-endian notation was used.}
\begin{equation}
\label{eq:tucker_mode_unfoldings}
    \Batrix{C}_{(k)} = \Matrix{B}_k \Batrix{A}_{(k)} (\Matrix{B}_1 \krp \cdots \krp \Matrix{B}_{k-1} \krp \Matrix{B}_{k+1} \krp \cdots \krp \Matrix{B}_d )^\trans.
\end{equation}
We define the \emph{Tucker rank} of $\Bensor{A}$ as the tuple of row-ranks of its mode unfoldings,
\begin{equation*}
    \Rank{T}{\Bensor{A}} = (\Rank{r}{\Batrix{A}_{(1)}}, \ldots, \Rank{r}{\Batrix{A}_{(d)}}),
\end{equation*}
which is consistent with the Tucker rank of ``usual'' tensors when $\Hilb = \F$.
Clearly, a componentwise inequality $\Rank{T}{\Bensor{A}} \preccurlyeq (n_1,\ldots,n_d)$ holds.
We say that a product
\begin{equation}
\label{eq:tucker}
    \Bensor{A} = \Bensor{G} \times_1 \Matrix{F}_1 \times_2 \cdots \times_d \Matrix{F}_d
\end{equation}
is a \emph{Tucker decomposition} of the function-valued tensor $\Bensor{A}$ with core $\Bensor{G} \in \Hilb^{r_1 \times \cdots \times r_d}$, factors $\Matrix{F}_s \in \F^{n_s \times r_s}$, and rank $(r_1, \ldots, r_d)$.

\begin{remark}
Alternatively, we could place the function-valued component of the decomposition into a factor, e.g., $\Bensor{A} = \Tensor{G} \times_1 \Batrix{F}_1 \times_2 \Matrix{F}_2 \times_3 \cdots \times_d \Matrix{F}_d$ with $\Tensor{G} \in \F^{r_1 \times \cdots \times r_d}$ and $\Batrix{F}_1 \in \Hilb^{n_1 \times r_1}$.
In this case, $r_1$ is the \emph{column-rank} of the first unfolding, while $r_k$ is the row-rank of the $k$th unfolding for $k > 1$.
Further decomposing $\Batrix{F}_1^\trans = \Batrix{W} \Matrix{Z}$ with $\Batrix{W} \in \Hilb^{r_1 \times \rho}$, we obtain a Tucker decomposition $\Bensor{A} = (\Tensor{G} \times_1 \Batrix{W}^\trans) \times_1 \Matrix{H}^\trans \times_2 \Matrix{F}_2 \times_3 \cdots \times_d \Matrix{F}_d$ in the sense of \cref{eq:tucker} with rank $(\rho, r_2, \ldots, r_d)$ and a structured function-valued core.
Because such a reduction is possible, we focus on the analysis of \cref{eq:tucker} in this paper.
\end{remark}

When the core of a Tucker decomposition is a subtensor of $\Bensor{A}$, we will call this decomposition \emph{Tucker-cross}.
There is a simple relationship between the Tucker rank of a function-valued tensor and the rank of its Tucker decomposition.

\begin{proposition}
\label{proposition:tucker_rank_properties}
If $\Bensor{A}$ admits a Tucker decomposition \cref{eq:tucker}, then $\Rank{T}{\Bensor{A}} \preccurlyeq \Rank{T}{\Bensor{G}}$. If this is a Tucker-cross decomposition, then $\Rank{T}{\Bensor{A}} = \Rank{T}{\Bensor{G}}$.
\end{proposition}
\begin{proof}
The inequality follows from \cref{eq:tucker_mode_unfoldings} and \Cref{lemma:ranks}. If $\Bensor{G}$ is a subtensor of $\Bensor{A}$, then \Cref{lemma:ranks} applied to every mode unfolding yields $\Rank{T}{\Bensor{G}} \preccurlyeq \Rank{T}{\Bensor{A}}$, so the Tucker ranks coincide. 
\end{proof}

The following result shows that there always exists a \emph{minimal} Tucker-cross decomposition, i.e., whose rank is equal to $\Rank{T}{\Bensor{A}}$.
We assume further throughout this section that $\Bensor{A} \neq 0$.

\begin{proposition}
\label{proposition:tucker_minimal_cross}
Let $\Bensor{A} \in \Hilb^{n_1 \times \cdots \times n_d}$ and $\Rank{T}{\Bensor{A}} = (r_1, \ldots, r_d)$. There exist index sets $\Index{I}_k \subseteq [n_k]$ of cardinality $|\Index{I}_k| = r_k$ and matrices $\Matrix{F}_k \in \F^{n_k \times r_k}$ of full column-rank such that
\begin{equation*}
    \Bensor{A} = \Bensor{A}(\Index{I}_1, \ldots, \Index{I}_d) \times_1 \Matrix{F}_1 \times_2 \cdots \times_d \Matrix{F}_d.
\end{equation*}
\end{proposition}
\begin{proof}
For every $k$, \Cref{lemma:ranks} shows that there exist $\Matrix{F}_k \in \F^{n_k \times r_k}$ of rank $r_k$ and $\Index{I}_k \subseteq [n_k]$ of cardinality $|\Index{I}_k| = r_k$ such that $\Batrix{A}_{(k)} = \Matrix{F}_k \Batrix{A}_{(k)}(\Index{I}_k,:)$. Note that $\Batrix{A}_{(k)}(\Index{I}_k,:) = \Matrix{S}_k \Batrix{A}_{(k)}$, where $\Matrix{S}_k = \Matrix{I}_{n_k}(\Index{I}_k,:)$ is a row-selection matrix, so $\Bensor{A} = \Bensor{A} \times_k \Matrix{F}_k \Matrix{S}_k$. Since $k$ is arbitrary,
\begin{equation*}
    \Bensor{A} = \Bensor{A} \times_1 \Matrix{F}_1 \Matrix{S}_1 \times_2 \cdots \times_d \Matrix{F}_d \Matrix{S}_d = \Bensor{A}(\Index{I}_1, \ldots, \Index{I}_d) \times_1 \Matrix{F}_1 \times_2 \cdots \times_d \Matrix{F}_d.
\end{equation*}
\end{proof}

\subsection{Tucker-cross approximation}
Now, we extend the cross approximation \cref{eq:cross_batrix} to function-valued tensors.
Consider $\Bensor{G} = \Bensor{A}(\Index{I}_1, \ldots, \Index{I}_d)$ at index sets $\Index{I}_k \subseteq [n_k]$.
For every $k$, we introduce a function-valued matrix
\begin{equation}
\label{eq:rows}
    \Batrix{R}_k = \Row{\Bensor{A}}{\Index{I}_1, \ldots, \Index{I}_d}{k} = \big[ \Bensor{A}(\Index{I}_1, \ldots, \Index{I}_{k-1}, :, \Index{I}_{k+1}, \ldots, \Index{I}_{d})\big]_{(k)}^\trans
\end{equation}
of size $\prod_{l \neq k} |\Index{I}_l| \times n_k$, whose rows consist of mode-$k$ fibers.
We then define the \emph{Tucker-cross approximation} of $\Bensor{A}$ at $(\Index{I}_1, \ldots, \Index{I}_d)$ as
\begin{equation}
\label{eq:cross_tucker}
    \Cross{\Bensor{A}}{\Index{I}_1, \ldots, \Index{I}_d}{T} = \Bensor{G} \times_1 \Big[ \big(\Batrix{G}_{(1)}^\trans\big)^\pinv \Batrix{R}_1 \Big]^\trans \times_2 \cdots \times_d \Big[ \big(\Batrix{G}_{(d)}^\trans\big)^\pinv \Batrix{R}_d \Big]^\trans \in \Hilb^{n_1 \times \cdots \times n_d}.
\end{equation}
When $d = 2$, this is exactly the cross approximation \cref{eq:cross_batrix}. 

\begin{theorem}
\label{theorem:tucker_cross_exact}
Let $\Bensor{B} = \Cross{\Bensor{A}}{\Index{I}_1, \ldots, \Index{I}_d}{T}$. Then
\begin{enumerate}
    \item $\Bensor{B}(\Index{I}_1, \ldots, \Index{I}_d) = \Bensor{G}$ and $\Rank{T}{\Bensor{B}} = \Rank{T}{\Bensor{G}}$;
    \item $\Row{\Bensor{B}}{\Index{I}_1, \ldots, \Index{I}_d}{k} = \Batrix{R}_k$ if and only if $\Rank{r}{\Batrix{G}_{(k)}} = \Rank{c}{\Batrix{R}_k}$;
    \item $\Bensor{B} = \Bensor{A}$ if and only if $\Rank{T}{\Bensor{G}} = \Rank{T}{\Bensor{A}}$.
\end{enumerate}
\end{theorem}
\begin{proof}
$\bm{1.}$ Note that $\Batrix{R}_k(:, \Index{I}_k) = \Batrix{G}_{(k)}^\trans$ and
\begin{equation*}
    \Bensor{B}(\Index{I}_1, \ldots, \Index{I}_d) = \Bensor{G} \times_1 \Big[ \big(\Batrix{G}_{(1)}^\trans\big)^\pinv \Batrix{G}_{(1)}^\trans \Big]^\trans \times_2 \cdots \times_d \Big[ \big(\Batrix{G}_{(d)}^\trans\big)^\pinv \Batrix{G}_{(d)}^\trans \Big]^\trans.
\end{equation*}
The equality $\Bensor{B}(\Index{I}_1, \ldots, \Index{I}_d) = \Bensor{G}$ follows from \cref{eq:tucker_mode_unfoldings} and the definition of the pseudoinverse. Then $\Cross{\Bensor{A}}{\Index{I}_1, \ldots, \Index{I}_d}{T}$ is a Tucker-cross decomposition of $\Bensor{B}$, so \Cref{proposition:tucker_rank_properties} applies.\\
\indent$\bm{2.}$ By definition \cref{eq:cross_tucker}, $\Row{\Bensor{B}}{\Index{I}_1, \ldots, \Index{I}_d}{k} = \Batrix{G}_{(k)}^\trans \big( \Batrix{G}_{(k)}^\trans \big)^\pinv \Batrix{R}_k$. By \Cref{lemma:ranks}, it holds $\Batrix{G}_{(k)}^\trans \big( \Batrix{G}_{(k)}^\trans \big)^\pinv \Batrix{R}_k = \Batrix{R}_k$ if and only if $\Rank{r}{\Batrix{G}_{(k)}} = \Rank{c}{\Batrix{R}_k}$.\\
\indent$\bm{3.}$ If $\Bensor{B} = \Bensor{A}$, see part one. If $\Rank{T}{\Bensor{G}} = \Rank{T}{\Bensor{A}}$, then \Cref{lemma:ranks} gives for every $k$ that
\begin{equation*}
    \Batrix{A}_{(k)} = \Big[ \big(\Batrix{G}_{(k)}^\trans\big)^\pinv \Batrix{R}_k \Big]^\trans \Matrix{S}_k \Batrix{A}_{(k)}, \quad \Bensor{A} = \Bensor{A} \times_k \Big[ \big(\Batrix{G}_{(k)}^\trans\big)^\pinv \Batrix{R}_k \Big]^\trans \Matrix{S}_k,
\end{equation*}
where $\Matrix{S}_k = \Matrix{I}_{n_k}(\Index{I}_k,:)$. Repeat this for each $k$ to finish the proof.
\end{proof}

\Cref{theorem:tucker_cross_exact} describes the interpolation properties of \cref{eq:cross_tucker}.
A Tucker-cross approximation of $\Bensor{A}$ recovers it exactly and it coincides with Tucker-cross decomposition whenever $\Bensor{A}(\Index{I}_1, \ldots, \Index{I}_d)$ preserves the Tucker rank.
Otherwise, the quality of Tucker-cross approximation depends on the choice of the index sets.

\subsection{Approximation guarantee for Tucker-cross}
For function-valued matrices, the corresponding approximation guarantee was derived in \cite[Thm.~7.11]{budzinskiy2025matrices}.
This result is based on the following one-sided approximation bound from \cite{budzinskiy2025matrices}, which is formulated in terms of the function-valued analogue of the Frobenius norm $\ell_2$:
\begin{equation*}
    \Norm{\Bensor{A}}{\ell_2(\Hilb)} = \left( \sum_{i_1 = 1}^{n_1} \cdots \sum_{i_d = 1}^{n_d} \Norm{\Bensor{A}(i_1, \ldots, i_d)}{\Hilb}^2 \right)^{1/2}.
\end{equation*}
We include it below for completeness and further usage.

\begin{theorem}
\label{theorem:cgr_error}
Let $\Batrix{A} = \Batrix{B} + \Batrix{E} \in \Hilb^{m \times n}$ with $\Batrix{B} \neq 0$, and let $\Index{I} \subseteq [m]$ and $\Index{J} \subseteq [n]$ satisfy $\Rank{c}{\Batrix{B}(\Index{I}, \Index{J})} = \Rank{c}{\Batrix{B}}$.
Let $\Batrix{B} = \Batrix{U} \Matrix{\Sigma} \Matrix{V}^\ast$ be an SVD and denote
\begin{equation*}
    \gamma^{(r)} = \Norm{\Batrix{U}(\Index{I}, :)^\pinv}{\ell_2(\Hilb) \to \ell_2}, \quad \gamma^{(c)} = \Norm{\Matrix{V}(\Index{J},:)^\pinv}{\ell_2 \to \ell_2}, \quad \gamma = \gamma^{(r)} + \gamma^{(c)} + 3 \gamma^{(r)} \gamma^{(c)}.
\end{equation*}
Then
\begin{equation*}
    \Norm{\Batrix{B} - \Batrix{A}(:, \Index{J}) \Batrix{A}(\Index{I}, \Index{J})^\pinv \Batrix{A}(\Index{I}, :)}{\ell_2(\Hilb)} \leq \gamma \Norm{\Batrix{E}}{\ell_2(\Hilb)} + O(\Norm{\Batrix{E}}{\ell_2(\Hilb)}^2).
\end{equation*}
\end{theorem}
When $\Batrix{E}$ is the trailing part of an SVD of $\Batrix{A}$, this theorem shows that one-sided cross approximation \cref{eq:precross_batrix} is \emph{quasi-optimal} in the $\ell_2(\Hilb)$ norm provided that the indices are ``informative'' enough to recover the truncated SVD of $\Batrix{A}$ itself.
A recursive application of \cref{theorem:cgr_error} leads to an error bound for function-valued tensors.

\begin{theorem}
\label{theorem:cross_error}
Let $\Bensor{A} = \Bensor{B} + \Bensor{E} \in \Hilb^{n_1 \times \cdots \times n_d}$ with $\Bensor{B} \neq 0$.
For each $k = 1, \ldots, d$, consider $\Index{I}_k \subseteq [n_k]$, let $\Batrix{B}_{(k)}^\trans = \Batrix{U}_{k} \Matrix{\Sigma}_k \Matrix{V}_k^\ast$ be an SVD of the transposed mode-$k$ unfolding, let $\Bensor{U}_k \in \Hilb^{n_1 \times \cdots \times \Rank{r}{\Batrix{B}_{(k)}} \times \cdots \times n_d}$ be tensorized $\Batrix{U}_k$, and denote
\begin{gather*}
    \gamma^{(r)}_k = \Norm{\Row{\Bensor{U}_k}{\Index{I}_1, \ldots, \Index{I}_d}{k}^\pinv}{\ell_2(\Hilb) \to \ell_2}, \quad \gamma^{(c)}_k = \Norm{\Matrix{V}_k(\Index{I}_k, :)^\pinv}{\ell_2 \to \ell_2}, \\
    \gamma_k = \gamma_k^{(r)} + \gamma_k^{(c)} + 3 \gamma_k^{(r)} \gamma_k^{(c)}.
\end{gather*}
Let $\Rank{T}{\Bensor{B}(\Index{I}_1, \ldots, \Index{I}_d)} = \Rank{T}{\Bensor{B}}$, then
\begin{equation*}
    \Norm{\Bensor{B} - \Cross{\Bensor{A}}{\Index{I}_1, \ldots, \Index{I}_d}{T}}{\ell_2(\Hilb)} \leq \bigg[ \gamma_1 + \sum_{k = 2}^d (\gamma_k + 1) \prod_{j=1}^{k-1} \gamma_j^{(c)} \bigg] \Norm{\Bensor{E}}{\ell_2(\Hilb)} + O(\Norm{\Bensor{E}}{\ell_2(\Hilb)}^2).
\end{equation*}
\end{theorem}
\begin{proof}
For the $k$th mode, denote by $\Matrix{S}_k \in \F^{|\Index{I}_k| \times n_k}$ the row-selection matrix and let $\Matrix{W}_k^\trans = \big(\Batrix{G}_{(k)}^\trans\big)^\pinv \Batrix{R}_k$.
Then $\Cross{\Bensor{A}}{\Index{I}_1, \ldots, \Index{I}_d}{T} = \Bensor{A} \times_1 \Matrix{Z}_1 \times_2 \cdots \times_d \Matrix{Z}_d $ with $\Matrix{Z}_k = \Matrix{W}_k \Matrix{S}_k$, and the error of Tucker-cross approximation admits a telescoping representation
\begin{equation*}
    \Bensor{D} = \Bensor{B} - \Cross{\Bensor{A}}{\Index{I}_1, \ldots, \Index{I}_d}{T} = (\Bensor{B} - \Bensor{A} \times_1 \Matrix{Z}_1) + \sum_{k = 2}^{d} (\Bensor{A} - \Bensor{A} \times_k \Matrix{Z}_k) \times_1 \Matrix{Z}_1 \times_2 \cdots \times_{k-1} \Matrix{Z}_{k-1}.
\end{equation*}
Applying \cref{theorem:cgr_error} to every unfolding, we get
\begin{gather*}
    \Norm{\Batrix{B}_{(1)} - \Matrix{Z}_1 \Batrix{A}_{(1)}}{\ell_2(\Hilb)} \leq \gamma_1 \Norm{\Bensor{E}}{\ell_2(\Hilb)} + O(\Norm{\Bensor{E}}{\ell_2(\Hilb)}^2), \\
    \Norm{\Batrix{A}_{(k)} - \Matrix{Z}_k \Batrix{A}_{(k)}}{\ell_2(\Hilb)} \leq (\gamma_k + 1) \Norm{\Bensor{E}}{\ell_2(\Hilb)} + O(\Norm{\Bensor{E}}{\ell_2(\Hilb)}^2), \quad 2 \leq k \leq d.
\end{gather*}
It follows that
\begin{equation*}
    \Norm{\Bensor{D}}{\ell_2(\Hilb)} \leq \bigg[ \gamma_1 + \sum_{k = 2}^d (\gamma_k + 1) \prod_{j=1}^{k-1} \Norm{\Matrix{Z}_j}{\ell_2 \to \ell_2} \bigg] \Norm{\Bensor{E}}{\ell_2(\Hilb)} + O(\Norm{\Bensor{E}}{\ell_2(\Hilb)}^2).
\end{equation*}
As shown in the proof of \cite[Thm.~7.11]{budzinskiy2025matrices}, it holds that
\begin{equation*}
    \Norm{\Matrix{Z}_j}{\ell_2 \to \ell_2} \leq \Norm{\Matrix{W}_j}{\ell_2 \to \ell_2} \leq \gamma_j^{(c)} + O(\Norm{\Bensor{E}}{\ell_2(\Hilb)}),
\end{equation*}
which establishes the sought bound.
\end{proof}

\Cref{theorem:cross_error} demonstrates that the quality of Tucker-cross approximation is determined by how well-conditioned the unfoldings of the selected subtensor are.
This result is similar to the error bound \cite[Thm.~7]{cai2021mode} for classical tensors.

\subsection{Higher-order SVD}
For comparison and evaluation of the Tucker-cross approximation quality, it is important to have a reference approximant.
For ordinary tensors, the HOSVD algorithm~\cite{de2000multilinear} computes a low-rank Tucker decomposition that provides a quasi-optimal approximation in the Frobenius norm $\ell_2$.

Suppose that the approximation rank satisfies $(r_1, \ldots, r_d) \preccurlyeq \Rank{T}{\Bensor{A}}$.
For every $k$, let $\Batrix{A}_{(k)}^\trans = \Batrix{U}_k \Matrix{\Sigma}_k \Matrix{V}_k^\ast$ be an SVD of the transposed mode-$k$ unfolding.
Denote by $\Matrix{V}_{k,r_k} \in \F^{n_k \times r_k}$ the first $r_k$ columns of $\Matrix{V}_{k}$ and by $\Matrix{\Sigma}_{k,r_k}$ the matrix obtained from $\Matrix{\Sigma}_k$ by setting all but its first $r_k$ diagonal entries to zero.
We define an HOSVD of $\Bensor{A}$ as a Tucker decomposition of rank $(r_1, \ldots, r_d)$:
\begin{align*}
    \hosvd{\Bensor{A}}{r_1, \ldots, r_d} &= \Bensor{G} \times_1 \overline{\Matrix{V}}_{1,r_1} \times_2 \cdots \times_d \overline{\Matrix{V}}_{d,r_d}, \\
    \Bensor{G} &= \Bensor{A} \times_1 \Matrix{V}_{1,r_1}^\trans \times_2 \cdots \times_d \Matrix{V}_{d,r_d}^\trans.
\end{align*}

\begin{theorem}
\label{theorem:hosvd}
Let $\Bensor{A} \in \Hilb^{n_1 \times \cdots \times n_d}$ and $r_1, \ldots, r_d \in \N$ be such that $(r_1, \ldots, r_d) \preccurlyeq \Rank{T}{\Bensor{A}}$. Let $\Bensor{B} = \hosvd{\Bensor{A}}{r_1, \ldots, r_d}$ and denote
\begin{equation*}
    \mathbb{M} = \set{\Bensor{C} \in \Hilb^{n_1 \times \cdots \times n_d}}{\Rank{T}{\Bensor{C}} \preccurlyeq (r_1, \ldots, r_d)}.    
\end{equation*}
Then
\begin{equation*}
    \Norm{\Bensor{A} - \Bensor{B}}{\ell_2(\Hilb)} \leq \bigg( \sum_{k = 1}^{d} \Norm{\Matrix{\Sigma}_k - \Matrix{\Sigma}_{k,r_k}}{\ell_2}^2 \bigg)^{1/2} \leq \sqrt{d} \min_{\Bensor{C} \in \mathbb{M}} \Norm{\Bensor{A} - \Bensor{C}}{\ell_2(\Hilb)}.
\end{equation*}
\end{theorem}
\begin{proof}
As in \Cref{remark:onb}, introduce an orthonormal basis $\{ \phi_{\lambda} \}_{\lambda = 1}^{\Lambda}$ in the linear span of the entries of $\Bensor{A}$ and consider its coefficient tensor $\Tensor{A} \in \F^{n_1 \times \cdots \times n_d \times \Lambda}$.
For each $k \in \{1,\ldots,d\}$, consider an SVD $\Batrix{A}_{(k)}^\trans = \Batrix{U}_k \Matrix{\Sigma}_k \Matrix{V}_k^\ast$ and let $\Tensor{U}^{[k]} \in \F^{N_k \times n_k \times \Lambda}$ with $N_k = \prod_{j \neq k} n_j$ be the coefficient tensor of $\Batrix{U}_k$.
Then
\begin{equation*}
    \Matrix{A}_{(k)} = \overline{\Matrix{V}}_k \Matrix{\Sigma}_k  \Matrix{U}^{[k]}_{(2)}
\end{equation*}
is also an SVD, and therefore the coefficient tensor $\Tensor{B}$ of $\Bensor{B}$, given by
\begin{equation*}
    \Tensor{B} = \Tensor{A} \times_1 \overline{\Matrix{V}}_{1,r_1} \Matrix{V}_{1,r_1}^\trans \times_2 \cdots \times_d \overline{\Matrix{V}}_{d,r_d} \Matrix{V}_{d,r_d}^\trans,
\end{equation*}
corresponds to the classical truncated HOSVD along the first $d$ modes, excluding the last mode of dimension $\Lambda$.
It follows from \cite[Property~10]{de2000multilinear} that
\begin{equation*}
    \Norm{\Bensor{A} - \Bensor{B}}{\ell_2(\Hilb)} = \Norm{\Tensor{A} - \Tensor{B}}{\ell_2} \leq \bigg( \sum_{k = 1}^{d} \Norm{\Matrix{\Sigma}_k - \Matrix{\Sigma}_{k,r_k}}{\ell_2}^2 \bigg)^{1/2},
\end{equation*}
and a similar coefficient-tensor argument yields the quasi-optimality bound.
\end{proof}

\subsection{Reduced-order modeling}
\label{subsec:tensor_rom}
We use Tucker-cross approximations \cref{eq:cross_tucker} to extend the low-rank ROM \cref{eq:lr_rom} from two to multiple parameters.
The setup of \Cref{subsec:rom_matrix} is repeated almost verbatim.
Consider a parameter-to-solution map $\eta : [0,1]^d \to \Hilb$, introduce one-dimensional grids $\alpha_{k,1}, \ldots, \alpha_{k, n_k}$ along each parametric dimension, and choose the corresponding interpolation basis functions $\varphi_{k,1}, \ldots, \varphi_{k, n_k}$.
The initial, ``implicit'' interpolation of $\eta$ is given by
\begin{align*}
    \eta(\alpha_1, \ldots, \alpha_d) &\approx \sum_{i_1 = 1}^{n_1} \cdots \sum_{i_d = 1}^{n_d} \eta(\alpha_{1,i_1}, \ldots, \alpha_{d,i_d}) \prod_{k = 1}^{d} \varphi_{k,i_k}(\alpha_k) \\
    &= \Bensor{A} \times_1 \Matrix{\upphi}_1(\alpha_1)^{\trans} \times_2 \cdots \times_d \Matrix{\upphi}_d(\alpha_d)^{\trans}
\end{align*}
with $\Bensor{A} \in \Hilb^{n_1 \times \cdots \times n_d}$ and $\Matrix{\upphi}_k(\alpha_k) = [\varphi_{k,1}(\alpha_k)~\cdots~\varphi_{k,n_k}(\alpha_k)]^{\trans}$.
The Tucker-cross approximation \eqref{eq:cross_tucker} of $\Bensor{A}$ gives
\begin{equation*}
    \Bensor{A} \approx \Bensor{A}(\Index{I}_1, \ldots, \Index{I}_d) \times_1 \Matrix{F}_1 \times_2 \cdots \times_d \Matrix{F}_d, \quad \Matrix{F}_k \in \Real^{n_k \times |I_k|},
\end{equation*}
which leads to the low-rank ROM
\begin{equation*}
    \hat{\eta}(\alpha_1, \ldots, \alpha_d) = \Bensor{A}(\Index{I}_1, \ldots, \Index{I}_d) \times_1 [\Matrix{\upphi}_1(\alpha_1)^{\trans} \Matrix{F}_1] \times_2 \cdots \times_d [\Matrix{\upphi}_d(\alpha_d)^{\trans} \Matrix{F}_d].
\end{equation*}
\Cref{theorem:tucker_cross_exact} implies that the ROM is interpolatory on a $I_1 \times \cdots \times I_d$ subgrid in the sense that
\begin{equation*}
    \hat{\eta}(\alpha_{1,i_1}, \ldots, \alpha_{d,i_d}) = \eta(\alpha_{1,i_1}, \ldots, \alpha_{d,i_d}), \quad (i_1, \ldots, i_d) \in I_1 \times \cdots \times I_d.
\end{equation*}
\section{Adaptive cross approximation of function-valued tensors}\label{sec:abc}
The algorithm we propose extends the algorithm of \cite{caiafa2010generalizing} from scalar-valued tensors to function-valued tensors.
Specifically, we extend the ACA-like algorithm for function-valued matrices proposed in~\cite{budzinskiy2025matrices} to function-valued tensors.
Its first ingredient is a \emph{rook-pivoting} procedure, which finds a large entry in a function-valued matrix by scanning its rows and columns (\Cref{alg:rook}).
Rook pivoting is de facto the standard way to select a new index pair in ACA \cite{dolgov2020parallel, shi2026distributed}.

\begin{algorithm2e}[ht]
\caption{Rook pivoting for function-valued matrices}
\label{alg:rook}
\DontPrintSemicolon
\KwIn{$\Batrix{A} \in \Hilb^{m \times n}$, starting column index $j_{*} \in [n]$, number of rounds $n_{\mathrm{rook}} \in \N_0$}
\For{$s = 1, \ldots, n_{\mathrm{rook}}$}{
    $i_* \gets \argmax_{i \in [m]} \Norm{\Batrix{A}(i, j_*)}{\Hilb}$\;
    $j_* \gets \argmax_{j \in [n]} \Norm{\Batrix{A}(i_*, j)}{\Hilb}$\;
}
\KwOut{new column index $j_*$}
\end{algorithm2e}

Assume that we wish to refine the cross approximation $\Bensor{B} = \Cross{\Bensor{A}}{\Index{I}_1, \ldots, \Index{I}_d}{T}$ of $\Bensor{A} \in \Hilb^{n_1 \times \cdots \times n_d}$ by expanding the index sets $\{ \Index{I}_k \}$.
The straightforward way is to apply rook pivoting to every unfolding of the residual $\Bensor{A} - \Bensor{B}$, but the corresponding complexity for the $k$th unfolding would be proportional to $n_k + \prod_{l \neq k} n_l$.
A standard way to reduce this complexity is to consider a subtensor. Specifically, given another collection of index sets $\{ \Index{I}_k' \}$, let us apply rook pivoting to $\Row{\Bensor{A} - \Bensor{B}}{\Index{I}_1', \ldots, \Index{I}_d'}{k}$ as defined in \cref{eq:rows}; the complexity is then proportional to $n_k + \prod_{l \neq k} |\Index{I}_l'|$.
We can now formulate the algorithm TuckerABC (\Cref{alg:abc}) of adaptive function-valued cross approximation in the Tucker format.\footnote{The name stands for Tucker Adaptive Bochner Cross.}

\begin{algorithm2e}[ht]
\caption{Adaptive function-valued cross approximation (TuckerABC)}
\label{alg:abc}
\DontPrintSemicolon
\KwIn{$\Bensor{A} \in \Hilb^{n_1 \times \cdots \times n_d}$, non-empty index sets $\{ \Index{I}_k' \}_{k = 1}^{d}$, number of iterations $n_{\mathrm{iter}} \in \N$, number of rook-pivoting rounds $n_{\mathrm{rook}} \in \N_0$}
$\Bensor{B} \gets 0, \Index{I}_1 \gets \emptyset, \ldots, \Index{I}_d \gets \emptyset$\;
\For{$s = 1, \ldots, n_{\mathrm{iter}}$}{
    \For{$k = 1, \ldots, d$}{
        $\Batrix{R}_k \gets \Row{\Bensor{A} - \Bensor{B}}{\Index{I}_1', \ldots, \Index{I}_d'}{k}$\;
        $j_k \gets \scup{draw}_k(n_k)$\tcc*{random or via predefined deterministic rule}
        $j_k \gets \scup{rook}(\Batrix{R}_k, j_k, n_{\mathrm{rook}})$\;
        $\Index{I}_k \gets \Index{I}_k \cup \{ j_k \}, \Index{I}_k' \gets \Index{I}_k' \cup \{ j_k \}$\;
    }
    $\Bensor{B} \gets \Cross{\Bensor{A}}{\Index{I}_1, \ldots, \Index{I}_d}{T}$\;
}
\KwOut{Tucker-cross approximation $\Bensor{B}$ such that $\Rank{T}{\Bensor{B}} \preccurlyeq (n_{\mathrm{iter}}, \ldots, n_{\mathrm{iter}})$}
\end{algorithm2e}

Note that the function-valued matrix $\Batrix{R}_k$ need not be formed explicitly in line 4, it is only required as an index-to-value map, which is evaluated $n_{\mathrm{rook}} (n_k + \prod_{l \neq k} |\Index{I}_l'|)$ times in line 6.
Each evaluation amounts to computing an entry of the Tucker-cross approximation $\Bensor{B}$ (i.e., a linear combination in $\Hilb$) and an entry of $\Bensor{A}$ (e.g., a solution of a PDE, which can be then cached in memory).
The geometry of the Hilbert space $\Hilb$ enters TuckerABC through lines 6 and 8.

Let us bound from above the number $N_{\mathrm{sol}}$ of PDE solves in \Cref{alg:abc}, assuming that every computed solution is immediately cached in memory. For simplicity, let $n = n_1 = \cdots = n_d$ and $t = |\Index{I}_1'| = \cdots = |\Index{I}_d'|$.
Then
\begin{align*}
    N_{\mathrm{sol}} &\leq \sum_{s = 1}^{n_{\mathrm{iter}}} d n_{\mathrm{rook}} \big(n + (t+s)^{d-1}\big) + dn n_{\mathrm{iter}}^{d-1} - (d-1) n_{\mathrm{iter}}^d \\
    &\leq d n n_{\mathrm{iter}} n_{\mathrm{rook}} + n_{\mathrm{rook}} (t + n_{\mathrm{iter}} + 1)^d + dn n_{\mathrm{iter}}^{d-1}.
\end{align*}
Thus, the fraction of $\Bensor{A}$ evaluated by \Cref{alg:abc} is at most
\begin{equation*}
    d n_{\mathrm{rook}} \frac{n_{\mathrm{iter}}}{n^{d-1}} + n_{\mathrm{rook}} \left( \frac{t + n_{\mathrm{iter}} + 1}{n} \right)^d + d \left( \frac{n_{\mathrm{iter}}}{n} \right)^{d-1},
\end{equation*}
which highlights the sample efficiency of the proposed low-rank ROM when the number of iterations remains small relative to the grid size in the parameter space.

To compute the Tucker-cross approximation in line 8, and specifically to compute the pseudoinverse of a function-valued matrix, we use the modified Gram--Schmidt (MGS) algorithm with column pivoting and partial reorthogonalization.
The numerical rank is determined based on the $\ell_2(\Hilb)$ norm of the currently selected column by comparing it with the machine epsilon in absolute and relative senses.
To apply the pseudoinverse, we compute the product with the adjoint of the orthogonal factor according to the formulas of MGS for numerical stability \cite[\S~2.4.4]{bjorck1996numerical}.

\subsection{Comparison with ACA for classical matrices and tensors}
\label{subsec:abc_vs_aca}
\paragraph{Lineage}
The proposed TuckerABC algorithm is a descendant of \cite[Alg.~1]{caiafa2010generalizing} that extends it to function-valued tensors.
The two algorithms build up structurally identical Tucker decompositions (compare \cref{eq:cross_tucker} with type-1 FSTD of \cite{caiafa2010generalizing}, also known as Chidori CUR \cite{cai2021mode}) and thus, e.g., both rely on \emph{long} inner products to compute factor matrices.
At the same time, the type-2 FSTD of \cite{caiafa2010generalizing}, a Tucker-cross decomposition of classical tensors that uses only \emph{short} inner products, does not admit an extension to function-valued tensors, because its derivation depends on the equivalence of column and row ranks of classical matrices.

\paragraph{Convergence}
There is not a general convergence theory for real-valued ACA-like algorithms with rook pivoting, predicting the rate of decay of the approximation error with iterations.
Rook pivoting fundamentally fails for sparse tensors if the nonzeros are never discovered, therefore an implicit applicability assumption is that the entries are sufficiently spread out across the tensor.
Just as in the case of real-valued tensors, the heuristic behind rook pivoting in \Cref{alg:abc} is that the selected submatrices remain well-conditioned throughout the iterations, which would then guarantee quasi-optimal convergence with respect to Tucker rank (\Cref{theorem:cross_error}; see also \Cref{sec:numer}).

Recent studies on matrix ACA establish its quasi-optimal convergence with rank when every entry is observed by means of complete pivoting \cite{gilles2026convergence} or random pivoting \cite{divan2026convergence}, provided the singular values decay at least algebraically.
Convergence rates have also been obtained for ACA with complete pivoting as formally applied to two-dimensional analytic \cite{townsend2015continuous, cortinovis2020maximum} and H\"older-continuous \cite{gilles2026convergence} functions.

\paragraph{Stopping criterion}
A heuristic stopping criterion for TuckerABC, just as for \cite[Alg.~1]{caiafa2010generalizing}, can be designed in two ways.
First, by sampling a subset of entries outside of the selected fibers at the end of each iteration and computing their approximation error.
Second, by estimating the distance between consecutive Tucker-cross approximants relative to the norm of the current approximant.
In real-valued matrix ACA, the latter is efficient thanks to the additive nature of cross-approximation updates, but this property does not extend to the cross approximation of tensors.
\section{Reduced-order modeling: Numerical experiments}\label{sec:numer}
In this section, we use TuckerABC to approximate the parameter-to-solution map of two \emph{nonlinear} parametric PDEs, as described in \Cref{subsec:tensor_rom}.
To assess the quality of the low-rank ROM, we consider parameter values only from the fine grid; that is, we compute the $\ell_2(\Hilb)$-norm approximation error for the function-valued tensor $\Bensor{A} \in \Hilb^{n_1 \times \cdots \times n_d}$.
Such an assessment requires forming the full tensor $\Bensor{A}$ and solving $\prod_{k} n_k$ PDEs.
The full tensor $\Bensor{A}$ is computed exclusively for the purpose of the error assessment, but it is not needed to \emph{construct} the low-rank ROM via TuckerABC.
The numerical PDEs are solved using FEniCS.\footnote{\url{https://fenicsproject.org/}}
To compute Sobolev-space inner products, we preliminary evaluate and store the sparse Gram matrix of a finite-element basis.
Consequently, the errors of numerical quadrature are etched into the inner-product geometry.

\subsection{Parametric nonlinear Stokes equations}
\label{subsec:stokes}
Consider the two-dimensional Stokes equations in $\Omega \subset [0,1]^2$, which describe a stationary flow of a non-Newtonian fluid with dominating viscous forces. Let $\Matrix{u} : \Omega \to \Real^2$ be the velocity and $p : \Omega \to \Real$ be the pressure.
When there are no external forces, the Stokes equations read
\begin{equation}
\label{eq:stokes}
    -\mathrm{div}(2 \nu \epsilon(\Matrix{u}) - p \Matrix{I}_2) = 0,~\mathrm{div}(\Matrix{u}) = 0~\text{ in }~\Omega,
\end{equation}
where $\epsilon(\Matrix{u}) = \tfrac{1}{2} \nabla \Matrix{u} + \tfrac{1}{2} (\nabla \Matrix{u})^\trans$ and $\nu$ is the viscosity.
The viscosity law
\begin{equation*}
    \nu_{\alpha,\beta}(\Matrix{u}) = \nu_0 + \alpha \Norm{\nabla \Matrix{u}}{\ell_2}^{\beta}.
\end{equation*}
involves the rheological parameters $\alpha\in [0,1]$ (consistency coefficient) and $\beta$ (flow index).
We consider a shear thickening fluid with $\beta\in[0,1]$.

The boundary is partitioned as $\partial\Omega = \partial\Omega_D \cup \partial\Omega_N$ with the no-slip and pressure-drop boundary conditions:
\begin{equation}
\label{eq:stokes_bcs}
    \Matrix{u} = 0~\text{ on }~\partial\Omega_D, \qquad 
    -(2 \nu \epsilon(\Matrix{u}) - p \Matrix{I}_2) \Matrix{n} = p_0 \Matrix{n}~\text{ on }~\partial\Omega_N,
\end{equation}
where $\Matrix{n} : \partial\Omega_N \to \Real^2$ is the outward normal, and $p_0 : \partial\Omega_N \to \Real$ defines the pressure drop that drives the flow.
We set 
\begin{equation*}
    p_0 = \gamma(1 - x)~\text{ with }~\gamma\in[1,5].    
\end{equation*}

We consider the flow domain $\Omega$ and its triangulation from~\cite{rognes2017fenics} (cf.~\Cref{fig:dolfin_solution}).
The system~\crefrange{eq:stokes}{eq:stokes_bcs} is discretized using Taylor--Hood finite elements on a locally refined, regular triangulation, resulting in $6495$ degrees of freedom.
The mesh and the corresponding solution for $\nu_0 = 0.01$, $\alpha = \beta = 0$, and $\gamma = 1$ are shown in \Cref{fig:dolfin_solution}.

\begin{figure}[ht!]
\centering
\begin{subfigure}[b]{0.328\linewidth}
\centering
	\includegraphics[width=\linewidth]{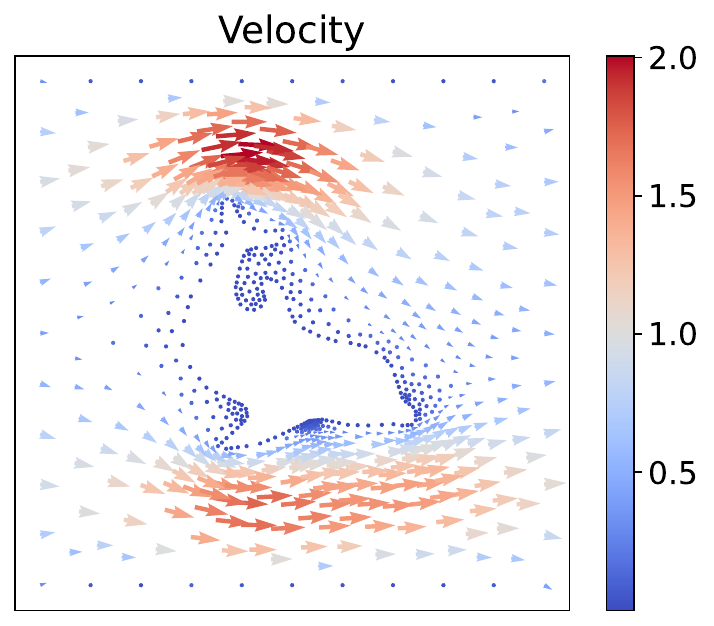}
\end{subfigure}%
\begin{subfigure}[b]{0.34\linewidth}
\centering
	\includegraphics[width=\linewidth]{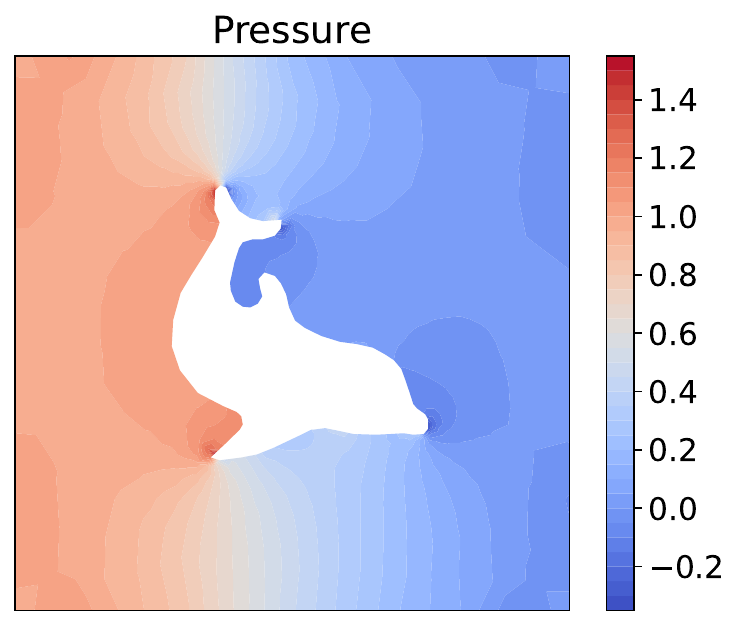}
\end{subfigure}%
\begin{subfigure}[b]{0.27\linewidth}
\centering
	\includegraphics[width=\linewidth]{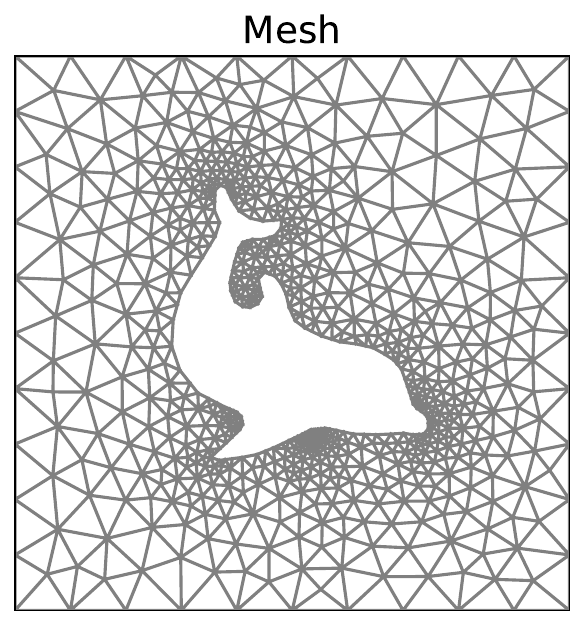}
\end{subfigure}
\caption{Numerical finite-element solution to \crefrange{eq:stokes}{eq:stokes_bcs} obtained with FEniCS.}
\label{fig:dolfin_solution}
\end{figure}

The standard POD-Galerkin ROM identifies a low-dimensional RB subspace with an $\Hilb^1(\Omega; \Real^3)$-orthonormal basis $\{ \xi_j \}_{j = 1}^{n_{\mathrm{RB}}}$ and then obtains approximate solutions to \cref{eq:stokes} within this subspace as $\sum_j c_j \xi_j$, computing the coefficients via the Galerkin-projected PDE---individually for every parameter triplet.
Thus, the PDE solutions are compressed in the physical space, but their variations with respect to parameters remain intact.
In contrast, TuckerABC is designed to approximate and compress this parametric dependence.
Moreover, it can be used both as a standalone ROM method (\Cref{subsec:monge}) and a recompression tool for an already existing POD-Galerkin ROM.

To showcase the latter, we introduce a uniform $64 \times 64 \times 64$ grid in the parameter space $[0,1] \times [0,1] \times [1,5]$ and construct an $n_{\mathrm{RB}}$-dimensional $\Hilb^1(\Omega; \Real^3)$-orthonormal POD basis for the entire collection of $64^3$ finite-element solutions to form a reference function-valued tensor
\begin{equation*}
    \Bensor{A} \in (\mathrm{span}\{ \xi_j \}_{j = 1}^{n_{\mathrm{RB}}})^{64 \times 64 \times 64} \subset \Hilb^1(\Omega; \Real^3)^{64 \times 64 \times 64}
\end{equation*}
of projected solutions.
Then, we approximate $\Bensor{A}$ with low Tucker rank using TuckerABC.
Because the computed basis is $\Hilb^1$-orthogonal, inner products between the entries of $\Bensor{A}$ simplify to the standard inner product of their coefficients in $\Real^{n_{\mathrm{RB}}}$, improving the efficiency of computations within TuckerABC (see \Cref{remark:onb}).

Although $\Bensor{A}$ is essentially equivalent to a real-valued tensor of size $n_{\mathrm{RB}} \times 64 \times 64 \times 64$, standard algorithms of cross approximation are not suitable for its approximation because of its vector-valued data-access model: to obtain a single entry, the entire vector of $n_{\mathrm{RB}}$ coefficients needs to be computed with the POD-Galerkin ROM, and most of them will remain unused.
Let us highlight that the block TT-cross algorithm from \cite{dolgov2015polynomial} is applicable in this setting; in fact, the initial RB stage is necessary to justify (see \Cref{sec:related_pde}) the usage of block TT-cross for parametric PDEs with finite-element numerical solutions (cf. \Cref{subsec:monge}).

We use this example as the testing ground to investigate how the hyperparameters of TuckerABC affect the low-rank ROM it produces, i.e., how well it approximates and compresses parametric dependencies.
The main hyperparameters are the number of rook-pivoting rounds $n_{\mathrm{rook}}$ and the initial auxiliary index sets $\{ \Index{I}_k' \}$, i.e., their cardinalities and content.
In \Cref{fig:dolfin_rb_params} (left), we set $n_{\mathrm{RB}} = 100$ and vary $n_{\mathrm{rook}} \in \{ 0, 1 \}$ and the cardinalities $|\Index{I}_k'| \in \{ 2, 8 \}$, selecting the index sets $\{ \Index{I}_k' \}$ and the initial guesses for rook pivoting uniformly at random.
Across 20 random seeds, we observe that
\begin{itemize}
    \item larger index sets $\{ \Index{I}_k' \}$ do not necessarily lead to reduced approximation errors,
    \item rook pivoting is critical for good performance of the algorithm, compared with random pivoting ($n_{\mathrm{rook}} = 0$).
    At the same time, taking $n_{\mathrm{rook}} > 1$ did not improve the approximation (we do not present the corresponding results).
\end{itemize}

Next, we look into \emph{how} the auxiliary index sets $\{ \Index{I}_k' \}$ are formed and change the probability distribution from which the $\Index{I}_k'$ are sampled at random; specifically, we consider the (approximate) \emph{leverage scores} of the unfoldings \cite{drineas2012fast}.
The leverage scores of a matrix $\Matrix{Q} \in \Real^{n \times r}$ with orthonormal columns are defined as the vector of squared row-norms
\begin{equation*}
    \Matrix{p} = \tfrac{1}{r} \begin{bmatrix}
        \Norm{\Matrix{Q}(1, :)}{2}^2 & \cdots & \Norm{\Matrix{Q}(n, :)}{2}^2
    \end{bmatrix}^\trans.
\end{equation*}
The vector $\Matrix{p}$ is nonnegative and satisfies $\Norm{\Matrix{p}}{1} = 1$, so it represents a probability distribution on the set of indices $[n]$.
For an unfolding $\Batrix{A}_{(k)} \in \Hilb^{n_k \times \prod_{l \neq k} n_l}$, we can take the SVD of its transpose, $\Batrix{A}_{(k)}^\trans = \Batrix{U} \Matrix{\Sigma} \Matrix{V}^\ast$, and the leverage scores of $\Matrix{V}$.
To approximate the leverage scores, we select $|\Index{I}_k'|$ columns of $\Batrix{A}_{(k)}$ uniformly at random and compute the SVD of the smaller $|\Index{I}_k'| \times n_k$ function-valued matrix.
These approximate leverage scores are then used to draw $\{ \Index{I}_k' \}$.
The results in \Cref{fig:dolfin_rb_params} (right) show that the use of leverage scores can lead to lower approximation errors in TuckerABC, at the cost of computing $\sum_{k = 1}^d n_k |\Index{I}_k'|$ additional entries of $\Bensor{A}$.

\begin{figure}[ht]
\centering
\begin{subfigure}[b]{0.5\linewidth}
\centering
	\includegraphics[width=\linewidth]{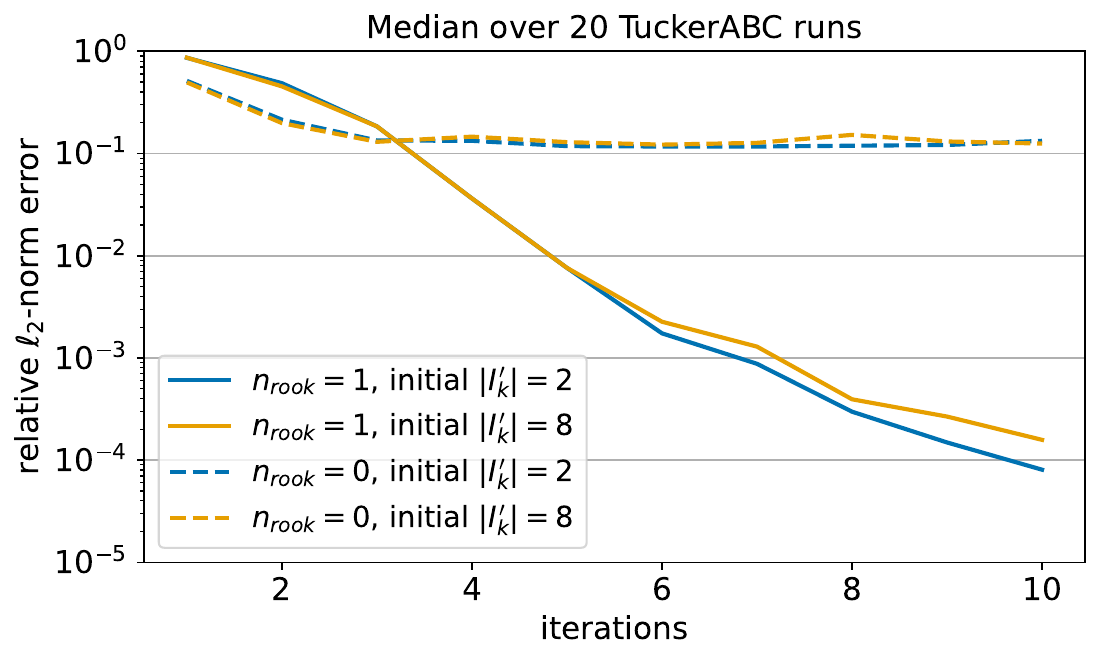}
\end{subfigure}%
\begin{subfigure}[b]{0.5\linewidth}
\centering
	\includegraphics[width=\linewidth]{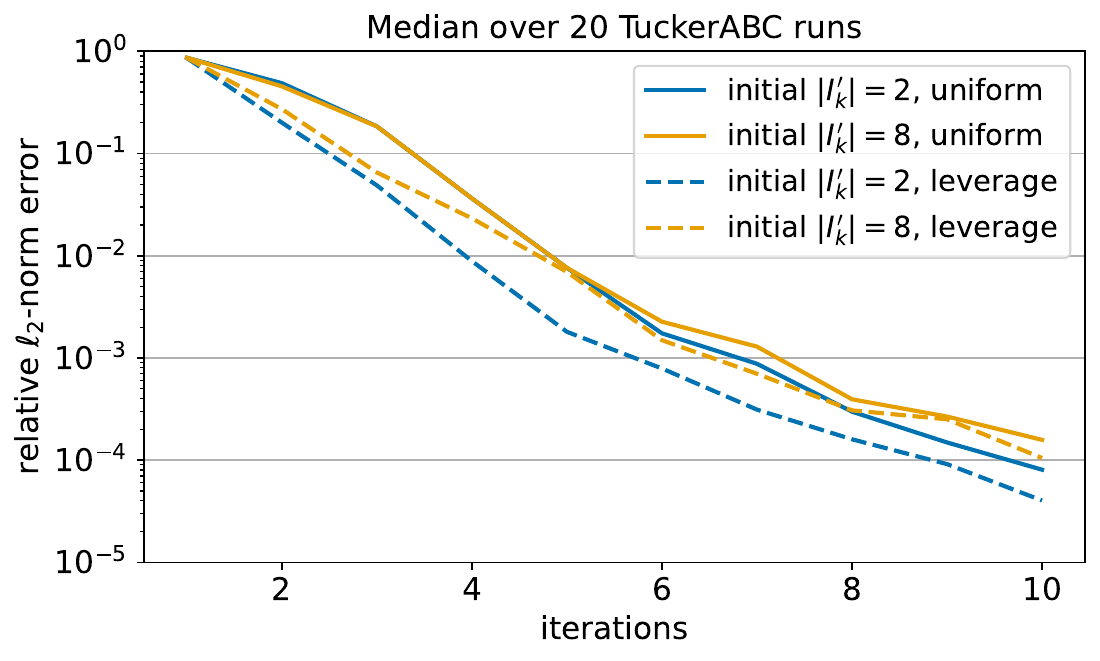}
\end{subfigure}
\caption{Relative $\ell_2(\Hilb^1)$-norm approximation errors of TuckerABC as applied to the function-valued tensor $\Bensor{A} \in (\mathrm{span}\{ \xi_j \}_{j = 1}^{n_{\mathrm{RB}}})^{64 \times 64 \times 64}$ of reduced-basis solutions to \crefrange{eq:stokes}{eq:stokes_bcs} for different hyperparameters of the algorithm.}
\label{fig:dolfin_rb_params}
\end{figure}

Finally, we compare the approximation error of TuckerABC with the quasi-optimal approximation error of HOSVD (\Cref{theorem:hosvd}), choosing one of the 20 TuckerABC runs that correspond to the blue dashed line in \Cref{fig:dolfin_rb_params} (right).
For every iteration of TuckerABC, we take the Tucker rank of its approximant and compute the truncated HOSVD with this rank.
\Cref{tab:dolfin_err} shows that the errors of TuckerABC are about two times higher than the errors of HOSVD (see a discussion of quasi-optimality in \Cref{subsec:abc_vs_aca}).
We can also observe that the parameter $\beta$, which determines the degree of nonlinearity in viscosity, is the most ``difficult'' for low-rank approximation.

\begin{table}[htb]
\centering
\caption{Comparison of relative $\ell_2(\Hilb^1)$-norm approximation errors of TuckerABC and HOSVD as applied to the function-valued tensor $\Bensor{A} \in (\mathrm{span}\{ \xi_j \}_{j = 1}^{n_{\mathrm{RB}}})^{64 \times 64 \times 64}$ of reduced-basis solutions to \crefrange{eq:stokes}{eq:stokes_bcs}.}
\label{tab:dolfin_err}
\begin{tabular*}{\textwidth}{@{\extracolsep\fill}lccc@{}}
    \toprule
    \multicolumn{2}{@{}c@{}}{TuckerABC} & \multicolumn{2}{@{}c@{}}{Approximation error} \\\cmidrule{1-2}\cmidrule{3-4}%
    Iterations & Rank & TuckerABC & HOSVD \\
    \midrule
    1  & (1, 1, 1) & $8.58 \cdot 10^{-1}$ & $4.66 \cdot 10^{-1}$ \\
    2  & (2, 2, 1) & $1.20 \cdot 10^{-1}$ & $6.88 \cdot 10^{-2}$ \\
    3  & (3, 3, 1) & $3.48 \cdot 10^{-2}$ & $1.83 \cdot 10^{-2}$ \\
    4  & (3, 4, 2) & $8.30 \cdot 10^{-3}$ & $4.00 \cdot 10^{-3}$ \\
    5  & (4, 5, 2) & $1.25 \cdot 10^{-3}$ & $7.98 \cdot 10^{-4}$ \\
    6  & (5, 6, 3) & $5.36 \cdot 10^{-4}$ & $2.36 \cdot 10^{-4}$ \\
    7  & (6, 7, 3) & $4.06 \cdot 10^{-4}$ & $1.07 \cdot 10^{-4}$ \\
    8  & (6, 8, 3) & $1.65 \cdot 10^{-4}$ & $6.73 \cdot 10^{-5}$ \\
    9  & (6, 9, 3) & $7.80 \cdot 10^{-5}$ & $3.77 \cdot 10^{-5}$ \\
    10 & (6, 10, 3) & $4.35 \cdot 10^{-5}$ & $2.35 \cdot 10^{-5}$ \\
    \bottomrule
\end{tabular*}
\end{table}

\subsection{Parametric Monge--Amp\`ere equation}
\label{subsec:monge}
The second example we study is a Monge--Amp\`ere equation in a square $\Omega = [-1,1]^2$ with homogeneous Dirichlet boundary condition and parametric right-hand side:
\begin{equation}
\label{eq:monge_ampere}
    \det(\nabla^2 u) = f(\Matrix{x}; \alpha, \beta, \gamma), \quad \Matrix{x}\in\Omega, \quad \left. u \right|_{\partial \Omega} = 0.
\end{equation}
Here, $\nabla^2 u$ denotes the $2 \times 2$ Hessian of the function $u : \Omega \to \Real$.
The highly nonlinear nature of \cref{eq:monge_ampere} necessitates special techniques for its numerical solution.
We use the vanishing moment method \cite{feng2009mixed} that consists in approximating \cref{eq:monge_ampere} with a fourth-order quasilinear equation with a small parameter $\varepsilon > 0$:
\begin{equation}
\label{eq:monge_ampere_vmm}
    -\varepsilon \Delta^2 u + \det(\nabla^2 u) = f(\Matrix{x}; \alpha, \beta, \gamma), \quad \Matrix{x}\in\Omega, \quad \left. u \right|_{\partial \Omega} = 0, \quad \left. \Matrix{n}^\trans (\nabla^2 u) \Matrix{n} \right|_{\partial \Omega} = \varepsilon,
\end{equation}
where $\Delta^2$ is the biharmonic operator and $\Matrix{n} : \partial\Omega \to \Real^2$ is the outward normal.
For every strictly positive $f$, there exists a unique strictly convex solution of \cref{eq:monge_ampere_vmm}.

Following \cite{feng2009mixed}, we discretize \cref{eq:monge_ampere_vmm} using Hermann--Miyoshi-type mixed finite elements, so that the solution $u$ is approximated with second-order finite elements.
We consider a family of Gaussian functions as the right-hand side,
\begin{equation*}
    f(\Matrix{x}; \alpha, \beta, \gamma) = \exp\left( -\frac{(x_1 - \alpha)^2 + (x_2 - \beta)^2}{\gamma} \right),
\end{equation*}
with parameters $(\alpha, \beta, \gamma)\in [-0.8, 0.8] \times [-0.8, 0.8] \times [0.001, 0.1]$.
A uniform triangulation of the square domain $\Omega$ with 1250 finite elements is used, yielding $2601$ degrees of freedom for the numerical solution $u$.
The mesh and the solution for $\varepsilon = 0.01$, $\alpha = 0.7$, $\beta = 0.6$, and $\gamma = 0.01$ are plotted in \Cref{fig:ma_solution}.

\begin{figure}[ht!]
\centering
\begin{subfigure}[b]{0.365\linewidth}
\centering
	\includegraphics[width=\linewidth]{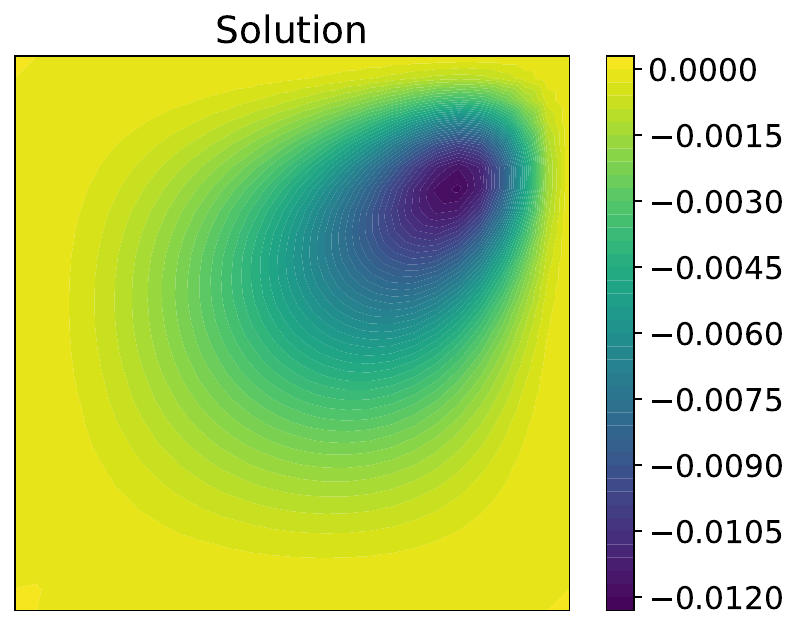}
\end{subfigure}%
\begin{subfigure}[b]{0.27\linewidth}
\centering
	\includegraphics[width=\linewidth]{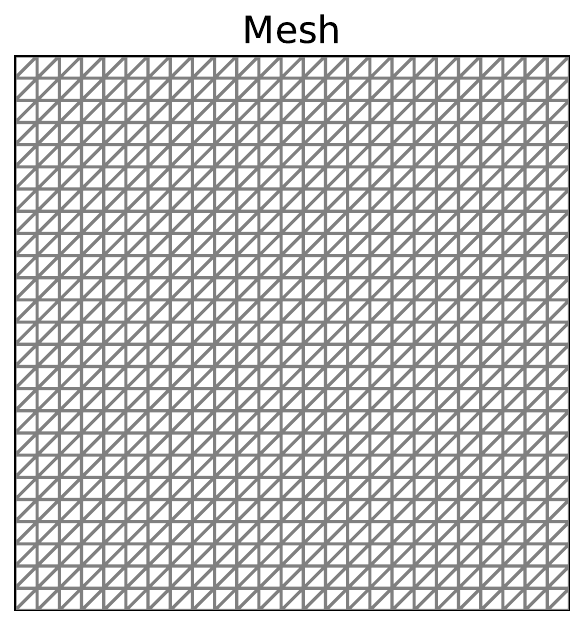}
\end{subfigure}
\caption{Numerical finite-element solution to \cref{eq:monge_ampere_vmm} obtained with FEniCS.}
\label{fig:ma_solution}
\end{figure}

With a uniform $50 \times 50 \times 50$ grid in the parameter domain we obtain $50 \times 50 \times 50$ solutions to \cref{eq:monge_ampere_vmm} with $\varepsilon = 0.01$ that we regard as a function-valued tensor $\Bensor{A}$ over the Sobolev space $\Hilb^1(\Omega)$.
To construct a low-rank ROM of the parameter-to-solution map, we apply TuckerABC to this function-valued tensor with $n_{\mathrm{rook}} = 1$ and auxiliary index sets $\{ \Index{I}_k' \}$ of cardinality 5 drawn uniformly at random.
The results in \Cref{tab:ma_err} correspond to a single random seed.

We use this example to assess how the approximation properties of the Tucker cross approximation~\cref{eq:cross_tucker} depend on the choice of inner product in the Hilbert space (\Cref{remark:ip}).
The first four columns in \Cref{tab:ma_err} have the same meaning as in \Cref{tab:dolfin_err}: they compare the $\ell_2(\Hilb^1)$-norm approximation errors obtained by TuckerABC with the quasi-optimal errors of HOSVD and show that the former are only $1.5$–$2.5$ times larger than the latter.
For the fifth column, we take the index sets selected by TuckerABC (the same as in the third column), compute the factors of~\cref{eq:cross_tucker} as if the entries of~$\Bensor{A}$ were vectors in~$\Real^{2601}$ equipped with the standard inner product, and then evaluate the resulting $\ell_2(\Hilb^1)$-norm approximation errors.
\Cref{tab:ma_err} shows that the quality of the inner-product-aware Tucker-cross approximation is consistently better after five or more iterations of TuckerABC.
However, the case with four iterations shows that this property is not universal, underscoring the importance of properly chosen indices.

\begin{table}[htb!]
\centering
\caption{Comparison of relative $\ell_2(\Hilb^{1})$-norm approximation errors of TuckerABC and HOSVD as applied to the function-valued tensor $\Bensor{A} \in (\Hilb^{1}(\Omega))^{50 \times 50 \times 50}$ of finite-element solutions to \cref{eq:monge_ampere_vmm}.}
\label{tab:ma_err}
\begin{tabular*}{\textwidth}{@{\extracolsep\fill}lcccc@{}}
    \toprule
    \multicolumn{2}{@{}c@{}}{TuckerABC} & \multicolumn{3}{@{}c@{}}{Approximation error} \\\cmidrule{1-2}\cmidrule{3-5}%
    Iterations & Rank & TuckerABC & HOSVD & TuckerABC ($\Real^{2601}$) \\
    \midrule
    4  & (4, 4, 1) & $8.93 \cdot 10^{-2}$ & $6.18 \cdot 10^{-2}$ & $8.83 \cdot 10^{-2}$ \\
    5  & (5, 5, 2) & $4.93 \cdot 10^{-2}$ & $2.05 \cdot 10^{-2}$ & $5.77 \cdot 10^{-2}$ \\
    6  & (6, 6, 3) & $1.84 \cdot 10^{-2}$ & $1.11 \cdot 10^{-2}$ & $2.36 \cdot 10^{-2}$ \\
    7  & (7, 7, 4) & $1.18 \cdot 10^{-2}$ & $6.01 \cdot 10^{-3}$ & $1.47 \cdot 10^{-2}$ \\
    8  & (8, 8, 5) & $6.80 \cdot 10^{-3}$ & $3.53 \cdot 10^{-3}$ & $7.75 \cdot 10^{-3}$ \\
    9  & (9, 9, 5) & $3.50 \cdot 10^{-3}$ & $2.05 \cdot 10^{-3}$ & $3.75 \cdot 10^{-3}$ \\
    10 & (10, 10, 6) & $2.04 \cdot 10^{-3}$ & $1.25 \cdot 10^{-3}$ & $2.25 \cdot 10^{-3}$ \\
    11 & (11, 11, 6) & $1.30 \cdot 10^{-3}$ & $7.75 \cdot 10^{-4}$ & $1.49 \cdot 10^{-3}$ \\
    12 & (12, 12, 7) & $8.52 \cdot 10^{-4}$ & $4.96 \cdot 10^{-4}$ & $9.00 \cdot 10^{-4}$ \\
    13 & (13, 13, 7) & $5.71 \cdot 10^{-4}$ & $3.26 \cdot 10^{-4}$ & $5.84 \cdot 10^{-4}$ \\
    14 & (14, 14, 8) & $4.14 \cdot 10^{-4}$ & $2.20 \cdot 10^{-4}$ & $4.28 \cdot 10^{-4}$ \\
    15 & (15, 15, 8) & $2.52 \cdot 10^{-4}$ & $1.51 \cdot 10^{-4}$ & $2.66 \cdot 10^{-4}$ \\
    \bottomrule
\end{tabular*}
\end{table}

In the next experiment, we investigate the stability of cross approximation produced by TuckerABC with respect to the spatial discretization of the PDE.
Since the framework of function-valued cross approximation \cref{eq:cross_tucker} applies to the exact solutions of \cref{eq:monge_ampere_vmm}, it is reasonable to expect that the rank and indices chosen by TuckerABC should stabilize for fine enough meshes.
To test this, we consider three more meshes similar to the one in \Cref{fig:ma_solution}, but with 128, 512, and 2048 triangular elements (they yield 289, 1089, and 4225 degrees of freedom for the numerical solution, respectively).
\Cref{tab:ma_mesh_refinement} shows close agreement of ranks and approximation errors on the finer meshes, while the 128-element mesh yields diverging results.
A closer inspection reveals that, although TuckerABC chooses the same rank for all four meshes after 10 iterations, the index sets $\Index{I}_2$ and $\Index{I}_3$ chosen for the coarsest mesh slightly differ from the rest (while all four $\Index{I}_1$ coincide); see \Cref{tab:ma_indices}.
Therefore, we observe that the performance of TuckerABC stabilizes with mesh refinement, as it becomes determined primarily by the (hypothetical) behavior of the algorithm on \emph{exact} PDE solutions.

\begin{table}[htb!]
\centering
\caption{Comparison of approximation ranks and relative $\ell_2(\Hilb^{1})$-norm approximation errors of TuckerABC as applied to the function-valued tensor $\Bensor{A} \in (\Hilb^{1}(\Omega))^{50 \times 50 \times 50}$ of finite-element solutions to \cref{eq:monge_ampere_vmm} on different spatial meshes.}
\label{tab:ma_mesh_refinement}
\begin{tabular*}{\textwidth}{@{\extracolsep\fill}lccc@{}}
    \toprule
    Iterations & Mesh & TuckerABC rank & TuckerABC error  \\
    \midrule
    10 & 128  & (10, 10, 6) & $3.17 \cdot 10^{-3}$ \\
       & 512  & (10, 10, 6) & $1.97 \cdot 10^{-3}$ \\
       & 1250 & (10, 10, 6) & $2.04 \cdot 10^{-3}$ \\
       & 2048 & (10, 10, 6) & $2.05 \cdot 10^{-3}$ \\
    \midrule
    15 & 128  & (14, 15, 7) & $2.15 \cdot 10^{-3}$ \\
       & 512  & (15, 15, 8) & $2.60 \cdot 10^{-4}$ \\
       & 1250 & (15, 15, 8) & $2.52 \cdot 10^{-4}$ \\
       & 2048 & (15, 15, 8) & $2.60 \cdot 10^{-4}$ \\
    \bottomrule
\end{tabular*}
\end{table}

\begin{table}[htb!]
\centering
\caption{Comparison of index sets chosen by TuckerABC after 10 iterations as applied to the function-valued tensor $\Bensor{A} \in (\Hilb^{1}(\Omega))^{50 \times 50 \times 50}$ of finite-element solutions to \cref{eq:monge_ampere_vmm} on different spatial meshes.}
\label{tab:ma_indices}
\begin{tabular*}{\textwidth}{@{\extracolsep\fill}lcc@{}}
    \toprule
    Mesh & $\Index{I}_2$ & $\Index{I}_3$ \\
    \midrule
    128  & $\{ 1, 6, 15, 19, 25, 29, 35, 40, 45, 50 \}$ & $\{ 1, 10, 12, 16, 20, 50 \}$ \\
    512  & $\{ 1, 6, 14, 19, 25, 29, 34, 40, 44, 49 \}$ & $\{ 6, 10, 11, 16, 19, 50 \}$ \\
    1250 & $\{ 1, 6, 14, 19, 25, 29, 34, 40, 44, 49 \}$ & $\{ 6, 10, 11, 16, 19, 50 \}$ \\
    2048 & $\{ 1, 6, 14, 19, 25, 29, 34, 40, 44, 49 \}$ & $\{ 6, 10, 11, 16, 19, 50 \}$ \\
    \bottomrule
\end{tabular*}
\end{table}

A simple idea follows from the observation made above.
To make the offline phase more efficient, one can construct the low-rank ROM of the parameter-to-solution map using a coarser mesh,
\begin{equation*}
    \Bensor{A}_{\mathrm{coarse}} \approx \Bensor{A}_{\mathrm{coarse}}(\Index{I}_1, \ldots, \Index{I}_d) \times_1 \Matrix{F}_1 \times_2 \cdots \times_d \Matrix{F}_d,
\end{equation*}
and reuse the selected indices and computed scalar factors (encoder) for finer meshes:
\begin{equation*}
    \Bensor{A}_{\mathrm{fine}} \approx \Bensor{A}_{\mathrm{fine}}(\Index{I}_1, \ldots, \Index{I}_d) \times_1 \Matrix{F}_1 \times_2 \cdots \times_d \Matrix{F}_d.
\end{equation*}
Only $\prod_{k} |\Index{I}_k|$ fine-mesh PDE solutions need to be evaluated in this case, and no extra computations are required.
We compare the effect of the coarse-mesh size on the resulting error corresponding to the mesh of size 2048.
The results in \Cref{tab:ma_reuse_coarse} suggest that once a mesh is fine enough to resolve the features of PDE solutions, it can be used to build a low-rank ROM that remains suitable for further mesh refinements.

Furthermore, we envisage the possibility of adaptively interleaving mesh refinements with TuckerABC iterations, so that finer-mesh cross approximation starts with the coarser-mesh indices instead of empty index sets.
This idea is left for future work.

\begin{table}[htb!]
\centering
\caption{Comparison of relative $\ell_2(\Hilb^{1})$-norm approximation errors of low-rank ROMs on a fine mesh built via 10 iterations of TuckerABC as applied to the function-valued tensor $\Bensor{A} \in (\Hilb^{1}(\Omega))^{50 \times 50 \times 50}$ of finite-element solutions to \cref{eq:monge_ampere_vmm} on coarser meshes.}
\label{tab:ma_reuse_coarse}
\begin{tabular*}{\textwidth}{@{\extracolsep\fill}lcccc@{}}
    \toprule
    & \multicolumn{4}{@{}c@{}}{Initial coarser mesh} \\\cmidrule{2-5}%
    & 128 & 512 & 1250 & 2048 \\
    \midrule
    Error & $3.3504 \cdot 10^{-3}$ & $2.0521 \cdot 10^{-3}$ & $2.0528 \cdot 10^{-3}$ & $2.0529 \cdot 10^{-3}$ \\
    \bottomrule
\end{tabular*}
\end{table}

\bibliographystyle{siamplain}
\bibliography{bib}

\end{document}